\documentclass[a4paper,12pt]{amsart}

\usepackage[includeheadfoot,twoside=False]{geometry}

\usepackage[active]{srcltx}
\usepackage{amsmath,amsthm,amsxtra}
\usepackage{amssymb}
\usepackage{mathabx}

\usepackage[mathscr]{eucal}

\usepackage{bm}
\usepackage[all]{xy}

\usepackage{aliascnt}

\theoremstyle{plain}
\newtheorem{thm}{Theorem}[section]
\newtheorem*{thm*}{Theorem}

\newaliascnt{prop}{thm}
\newaliascnt{cor}{thm}
\newaliascnt{lem}{thm}
\newaliascnt{claim}{thm}
\newaliascnt{defn}{thm}
\newaliascnt{ques}{thm}
\newaliascnt{conj}{thm}
\newaliascnt{fact}{thm}
\newaliascnt{rem}{thm}
\newaliascnt{ex}{thm}
\newtheorem{prop}[prop]{Proposition}
\newtheorem{cor}[cor]{Corollary}
\newtheorem{lem}[lem]{Lemma}
\newtheorem{claim}[claim]{Claim}
\newtheorem*{prop*}{Proposition}
\newtheorem*{cor*}{Corollary}
\newtheorem*{lem*}{Lemma}
\newtheorem*{claim*}{Claim}
\theoremstyle{definition}
\newtheorem{defn}[defn]{Definition}

\newtheorem*{defn*}{Definition}
\newtheorem*{ques*}{Question}
\newtheorem*{conj*}{Conjecture}

\newtheorem*{prob*}{Problem}

\newtheorem{rem}[rem]{Remark}
\newtheorem{ex}[ex]{Example}
\newtheorem*{fact*}{Fact}
\newtheorem*{rem*}{Remark}
\newtheorem*{ex*}{Example}
\aliascntresetthe{prop}
\aliascntresetthe{cor}
\aliascntresetthe{lem}
\aliascntresetthe{claim}
\aliascntresetthe{defn}
\aliascntresetthe{ques}
\aliascntresetthe{conj}
\aliascntresetthe{fact}
\aliascntresetthe{rem}
\aliascntresetthe{ex}
\usepackage[pointedenum]{paralist}
\usepackage{varioref}
\labelformat{equation}{\textnormal{(#1)}}
\labelformat{enumi}{\textnormal{(#1)}}

\usepackage[a4paper]{hyperref}

\def\textsectionN~{\textsection{}}

\newcommand{\aij}{a_i^j}
\newcommand{\sU}{\mathcal{U}}

\newcommand{\Go}{\mathbb{G}\spcirc}
\newcommand{\Xo}{X\spcirc}

\newcommand{\sQ}{\mathscr{Q}}
\newcommand{\sS}{\mathscr{S}}
\newcommand{\sHom}{\mathscr{H}om}

\newcommand\Ri{0 \leq i \leq n}
\newcommand\Rk{1 \leq k \leq n}
\newcommand\Rj{n+1 \leq j \leq N}
\newcommand\Rij{\Ri,\Rj}

\newcommand{\Zo}{{Z\spcirc}}

\renewcommand\phi{\varphi}
\renewcommand\epsilon{\varepsilon}
\renewcommand\leq{\leqslant}
\renewcommand\geq{\geqslant}

\makeatletter
\newcommand{\set}{%
  \@ifstar{\@setstar}{\@set}%
}%
\newcommand{\@setstar}[2]{\{\, #1 \mid #2 \,\}}
\newcommand{\@set}[1]{\{ #1 \}}
\makeatother

\newcommand{\trans}[1][1]{\raisebox{#1ex}{\scriptsize\kern0.1em$t$\kern-0.1em}}

\newcommand{\PP}{\mathbb{P}}
\newcommand{\cP}{{\PP_{\!\! *}}}
\newcommand{\TT}{\mathbb{T}}
\newcommand{\PN}{\PP^N}
\newcommand{\Pn}{\PP^n}
\newcommand{\GN}{\Gr(n, \PP^N)}

\newcommand{\A}{\mathbb{A}}
\newcommand{\sO}{\mathscr{O}}

\newcommand{\textgene}[1]{\ \ \text{#1}\ \,}

\newcommand{\textand}{\textgene{and}}

\DeclareMathOperator{\Frob}{Frob}
\DeclareMathOperator{\rk}{rk}%
\DeclareMathOperator{\Hom}{Hom}%
\DeclareMathOperator{\Spec}{Spec}

\newcommand{\Gr}{\mathbb{G}}
\newcommand\spcirc{^\circ}
\newcommand{\tdiff}[2]{{\partial #1}/{\partial #2}}
\newcommand{\diff}{\tdiff}

\newcommand{\kk}{\Bbbk}

\DeclareMathOperator{\rank}{rk}

\DeclareMathOperator{\id}{id}

\DeclareMathOperator{\pr}{pr}
\DeclareMathOperator{\chara}{char}

\DeclareMathOperator{\bir}{bir}

\def\Q{\mathbb{Q}}
\def\R{\mathbb{R}}
\def\C{\mathbb{C}}
\def\A{\mathbb{A}}

\def\r+{\mathbb{R}_{\geq 0}}

\def\ep{\varepsilon}

\def\r+{{\R}_{\geq 0}}
\def\q+{{\Q}_{\geq 0}}
\def\P{\mathbb{P}}
\def\arw{\rightarrow}
\def\*c{\C^{\times}}

\def\trdeg{\mathrm{tr.deg}}

\def\A{\mathbb {A}}

\def\C{\mathbb {C}}

\def\G{\mathbb {G}}

\def\Q{\mathbb {Q}}
\def\R{\mathbb {R}}

\def\T{\mathbb {T}}

\newcommand{\calf}{\mathcal {F}}
\newcommand{\calg}{\mathcal {G}}

\newcommand{\calu}{\mathcal {U}}

\makeatletter
  
  \@addtoreset{equation}{section}
\makeatother

\title[On Gauss maps in positive characteristic]{On Gauss maps in positive characteristic\\
  in view of images, fibers, and field extensions
}%

\author[K.~Furukawa]{Katsuhisa~FURUKAWA}
\address{
Department of Mathematics, Center for Advanced Studies in %
Theoretical Sciences, National Taiwan University, %
Taipei, 
Taiwan}
\email{katu@tims.ntu.edu.tw}

\author[A.~Ito]{Atsushi~Ito}
\address{Department of Mathematics, 
Kyoto University,
Kyoto, Japan}
\email{aito@math.kyoto-u.ac.jp}

\subjclass[2010]{14N05}
\keywords{Gauss map}

\begin{document}

\maketitle

\begin{abstract}
  The Gauss map of a projective variety $X \subset \PN$
  is a rational map from $X$ to a Grassmann variety.   In positive characteristic, we show the following results.
  (1) For given projective varieties $F$ and $Y$, we construct a projective variety $X$ whose Gauss map has
  $F$ as its general fiber and has $Y$ as its image.
  More generally, we give such construction for families of varieties over $Y$ instead of fixed $F$.
  (2)  At least in the case when the characteristic is not equal to $2$,
  any inseparable field extension
  appears as the extension induced from the Gauss map of some $X$.
\end{abstract}

\section{Introduction}
\label{sec:introduction}

Let $X \subset \PN$ be an $n$-dimensional projective variety over an algebraically closed field $\Bbbk$ of arbitrary characteristic $p \geq 0$.
The \emph{Gauss map} $\gamma$ of $X$
is defined as the rational map
$\gamma = \gamma_X: X \dashrightarrow \Gr(n, \PN)$
which maps each smooth point $x \in X$ to the embedded tangent space $\TT_xX$ to $X$ at $x$ in $\PN$.

J.~M.~Landsberg and J.~Piontkowski independently
characterized \emph{images} of Gauss maps in characteristic zero;
they gave an equivalent condition for a subvariety of a Grassmann variety $\GN$
to be the image of some Gauss map $\gamma$
(see \cite[2.4.7]{FP} and \cite[Theorem~3.4.8]{IL}).
As a generalization, a characterization of images of \emph{separable} Gauss maps was given
in positive characteristic \cite[Theorem~1.2]{expshr} (see also \autoref{thm:char-sepa}).
From these results,
a subvariety $Y \subset \GN$ with $\dim Y \leq n$ cannot be the image of a separable Gauss map
in general, and if $Y$ is the image of separable $\gamma_X$ for some $X \subset \PN$, then such $X$ is uniquely determined by $Y$.

On the other hand, the following was known about images of \emph{inseparable} Gauss maps of curves.
A.~H.~Wallace \cite[\textsection{}7]{Wallace},
S.~L.~Kleiman and A.~Thorup \cite[I-3]{Kleiman86}
showed that,
for any plane curve $Y \subset (\PP^2)\spcheck = \Gr(1, \PP^2)$ in $p > 0$
and for any positive integers $s$ and $r$,
there exists a curve $X \subset \PP^2$
such that
$\overline{\gamma_X(X)} =Y$
and the separable and inseparable degrees of $\gamma_X$ are $s$ and $p^r$
(in particular, there are infinitely many $X$'s such that $\overline{\gamma_X(X)} =Y$ for fixed $Y$ by taking various degrees).
More generally,
H.~Kaji \cite[\textsection{}4, Proposition]{Kaji1993} showed that,
for any curve $Y \subset \G(1,\PP^N)$ in $p > 0$ and for any finite inseparable field extension $L / K(Y)$,
there exists a curve $X \subset \PP^N$ with $K(X)=L$ such that
$\overline{\gamma_X(X)} =Y$ and the extension $K(X)/K(Y)$ induced by $\gamma_X$ coincides with the given $L /K(Y)$.

We investigate higher-dimensional subvarieties of $\GN$.
Throughout this paper,
by a \emph{field} we shall mean a finitely generated field over $\kk$.
For a field extension $L/K$, we denote by $\delta_{L/K}$ the natural $L$-linear map
$\Omega_{K/\kk} \otimes_{K} L \rightarrow \Omega_{L/\kk}$.

\begin{thm}\label{thm:anyY}
  Assume $p = \chara\kk > 0$.
  Let $Y \subset \G(n,\PP^N)$ be a projective subvariety with $1 \leq \dim Y \leq n$.
  For any field extension $L / K(Y)$ such that $\trdeg_{\kk} L =n$ and $\delta_{L/K(Y)}$ is the zero map,
  there exists an $n$-dimensional subvariety $X \subset \PP^N$
  with $K(X)=L$ such that $\overline{\gamma_X(X)} =Y$ holds and the extension $K(X)/ K(Y)$ induced by $\gamma_X$ coincides with the given $L / K(Y)$.

  In particular, any $Y$ appears as the image of the Gauss map of some $X \subset \PN$,
  and there are infinitely many choices of such $X$
  (by taking various extensions $L$'s over $K(Y)^{1/p}$).
\end{thm}

When $Y$ is a curve,
$\delta_{L/K(Y)} = 0$ holds if and only if $L/K(Y)$ is inseparable.
Hence the above Kaji's result
is nothing but the case when $n=\dim Y =1$ in \autoref{thm:anyY}.

\vspace{2mm}
In \autoref{thm:anyY},
we assume $\delta_{L/K(Y)} =0$ for the field extension $L/K(Y)$.
Then, how about the case when $\delta_{L/K(Y)} \neq 0$?
If $L/K(Y)$ is inseparable but $\delta_{L/K(Y)} \neq 0$,
the statement in \autoref{thm:anyY} does not hold in general,
even if $L/K(Y)$ is finite (see \autoref{ex_Y_with_r>1}).

On the other hand,
S.~Fukasawa and Kaji \cite{fukaji2010} showed that
for any field $L$ and any integer $0 \leq r \leq n:=\trdeg_{\kk} L$
($r \neq 1$ if $p=2$), there exists 
a projective variety $X \subset \PN$ with $K(X) = L$ such that
the rank of the Gauss map $\gamma$
(i.e., the rank of the $L$-linear map $\delta_{L/K(\gamma(X))}$ for the extension $L/K(\gamma(X))$ induced by $\gamma$)
is equal to the given integer $r$.

We refine their result at least in the case when $p \geq 3$. In the following theorem,
we construct a projective variety $X$ whose Gauss map induces a given inseparable extension $L/K$,
not only has a given rank.

\begin{thm}\label{thm_construction_for_insep}
  Let $L / K$ be an inseparable field extension
  with $\trdeg_{\kk} K \geq 1$
  and $p >0$,
  and let $N$ be an integer with $N > n:= \trdeg_{\kk}L$.
  Assume that $p \geq 3$ or $\rk_L (\delta_{L/K})$ is even. Then
  there exists 
  an $n$-dimensional non-degenerate projective variety $X \subset \PP^{N}$
  with $K(X)=L$ such that 
  the extension $K(X) / K(\gamma_X(X))$ induced by $\gamma_X$ coincides with the given $L/K$.
\end{thm}

The essential part of the proof of this theorem
is to show the existence of a hypersurface $X \subset \PP^{n+1}$ satisfying the above condition (see \autoref{thm_construction_for_insep:rem}).
However, in the case when $p=2$ and $\rk_L (\delta_{L/K})$ is odd,
there is no hypersurface $X \subset \PP^{n+1}$ satisfying the condition
(see \cite[Remark 5.3]{gmaptoric}, for example).
In this case,
we do not know whether there exists a subvariety $X \subset \PN$ with $N \geq n+2$ 
whose Gauss map induces the given extension $L/K$.
\\

Finally we focus on general fibers of Gauss maps.
In $p >0$,
Fukasawa \cite[Theorem 1]{Fukasawa2006}
showed that \emph{any} projective variety $F$ appears as a \emph{general fiber} of some inseparable Gauss map,
that is,
he gives a construction of a projective variety $X$
such that a general fiber of $\gamma_X$ with the reduced structure is isomorphic to the given $F$.

Instead of fixed $F$,
we consider a (not necessarily flat) family $\calf = \set{F_y}_{y \in Y}$ of projective varieties over a subvariety $Y \subset \GN$ in $p > 0$.
Then there exists $X \subset \PN$ such that $\overline{\gamma_X(X)}=Y$ and 
the fiber $\gamma_X^{-1}(y)$ over general $y \in Y$ with the reduced structure is isomorphic to $F_y$.
The precise statement is as follows:

\begin{thm}\label{thm:family}
  Assume $p >0$.
  Let $Y \subset \G(n,\PP^N)$ be a projective variety with $\dim Y \geq 1$
  and let $\calf \subset Y \times \PP^{N'}  $ be an $n$-dimensional projective variety such that the first projection $f: \calf \rightarrow Y$
  is surjective.
  Then there exist an $n$-dimensional projective variety $X \subset \PP^N$
  and a generically bijective
  rational map $h : X \dashrightarrow \calf$
  such that $\gamma_X$ is equal to $ f \circ h$.
  In particular, $\overline{\gamma_X(X)} = Y$ holds.
  
  Furthermore, if we assume $n \geq N'$,
  we can take $X \subset \PP^N$ such that the fiber ${\overline{\gamma_X^{-1}(y)}}_{red} \subset \PP^N$ of $\gamma_X$ over general $y \in Y$ is projectively equivalent to
  $F_y:= f^{-1}(y)_{red} \subset  \{y\} \times \PP^{N'}$.
  \begin{equation*}
    \begin{split}
      \xymatrix{        X  \ar@<-0.5ex>@{-->}[rrd]_(.45){\gamma_X} \ar@{-->}[rr]^{h}_{\kern2.5ex\mbox{\scriptsize gen.\,bij.}} && \calf \ar@{->>}[d]^{f} \ar@{}[r]|{\kern-2ex\mbox{$\subset$}} & Y \times \PP^{N'}
        \\
        && Y \ar@{}[r]|{\kern-2ex\mbox{$\subset$}}& \GN \makebox[0pt]{\ .}
      }    \end{split}
  \end{equation*}
\end{thm}

Roughly,
\autoref{thm:family} states that
any surjective morphism $\calf \arw Y$ appears as a Gauss map up to generically bijective rational maps.
We note that the assumption $n \geq N'$ is necessary in the last statement of \autoref{thm:family} 
since for $y =[L ] \in Y$,
the general fiber ${\overline{\gamma_X^{-1}(y)}}_{red}$ is contained in $L \simeq \PP^n$.

In contrast to Fukasawa's result, we can construct $X$ such that general fibers of $\gamma_X$ with reduced structures are not isomorphic each other.
For example, such $X$ is given by taking $\calf \rightarrow Y$ to be an elliptic fibration with $F_{y} \not\simeq F_{y'}$ at general $y$ and $y'$.
Fukasawa's result is recovered by taking any $(N'-\dim(F))$-dimensional $Y \subset \G(N',\PP^N)$ and setting $\calf = Y \times F$ for a given projective variety $F \subset \PP^{N'}$.
For more discussion, see \autoref{thm:family:rem}.

The rank of the Gauss map of $X$ which we construct in the proof of \autoref{thm:family} is zero.
By combining \autoref{thm:family} and  \autoref{thm_fib_r_c},
we can construct $X$ such that $\gamma_X$ has a given variety as a general
fiber and the rank of $\gamma_X$ is a given positive integer
(see \autoref{prop_rank:rem}).

In the previous paper \cite{gmaptoric},
the authors investigated Gauss maps of toric varieties,
and \autoref{thm:family} and \autoref{thm_fib_r_c} give extensions of results in \cite[\textsection{}5]{gmaptoric} to non-toric cases.
\\

This paper is organized as follows.
In \textsection{}2, we extend the definition of shrinking maps and
give a characterization of graphs of Gauss maps in any characteristic.
In \textsection{}3, by investigating numbers of generators of inseparable field extensions and applying the characterization in \textsection{}2,
we prove \autoref{thm:anyY} and \autoref{thm_construction_for_insep}.
In \textsection{}4, we show \autoref{thm:family}. In \textsection{}5, we describe {degeneracy maps} in the context of shrinking maps
by using the second fundamental form.
As a corollary,
we recover the main theorem of the first author's paper \cite{expshr}.

\subsection*{Acknowledgments}

The authors would like to thank Professors
Satoru Fukasawa and Hajime Kaji for their valuable comments and advice.
The first author was partially supported by JSPS KAKENHI Grant Number 25800030.
The second author was supported by the Grant-in-Aid for JSPS fellows, No.\ 26--1881.

\section{Subvarieties of the universal families of Grassmann varieties}
\label{sec:subv-univ-family}

In this section, we work over an algebraically closed field $\kk$ of any characteristic.
Let us consider the universal family of a Grassmann variety $\GN$,
\begin{equation}\label{eq:duf-univ}
  \sU = \sU_{\GN} :=\left\{ ([L],x) \in  \GN \times \PN \, | \, x \in L \right\} \subset \GN \times \PN.
\end{equation}
Let $X \subset \PP^N$ be an $n$-dimensional projective variety.
Then the \emph{graph} of the Gauss map $\gamma_X$ of $X$, which is the closure of the image of
\[
(\gamma_X, \id_X) : X \dashrightarrow \G(n, \PP^N) \times \PP^N,
\]
is contained in $\sU$ since $x \in \T_x X =\gamma_X(x)$.

As we mentioned in \textsection{}1,
Landsberg and Piontkowski independently gave an necessary and sufficient condition for a subvariety $Y \subset \G(n,\PP^N)$
to be the image of the Gauss map $\gamma_X$ of some $X \subset \PP^N$ in $\chara \kk=0$
by using the \emph{shrinking map} of $Y$.
In addition, a generalization to images of separable $\gamma_X$ was given by the first author
in $\chara \kk \geq 0$.
However, it seems that such characterization does not work well when we consider images of
\emph{inseparable} $\gamma_X$.  

In order to analyze (possibly inseparable) Gauss maps $\gamma_X$,
we characterize not images but graphs of $\gamma_X$ by generalizing the method of shrinking maps.
First, we define the shrinking map as follows.

\vspace{2mm}
We denote by $\sQ$ and $\sS$
the universal quotient bundle and subbundle of rank $n+1$ and $N-n$ on $\GN$
with the exact sequence
\begin{equation}\label{eq_univ_seq_U}
  0 \rightarrow \sS \rightarrow H^0(\PN, \sO(1)) \otimes \sO_{\GN}\rightarrow \sQ \rightarrow 0.
\end{equation}
Note that
$\sU = \cP(\sQ\spcheck)$ holds,
where $\cP (\mathscr{A}) := \mathbf{Proj}(\bigoplus_d S^d (\mathscr{A}\spcheck))$ for a  locally free sheaf $\mathscr{A}$.
Hence we have the tautological invertible sheaf $\sO_{\calu}(-1)$ on $\sU$,
which is a subsheaf of the pull-back of $\sQ\spcheck$ under the first projection $\sU \rightarrow \GN$.

\begin{defn}\label{thm:def-shr-Zf}
  Let $f: Z \dashrightarrow \GN$ be a rational map from a variety $Z$.
  We define the \emph{shrinking map of $Z$ with respect to $f$}
  \[
  \sigma = \sigma_{Z, f} : Z \dashrightarrow \Gr(n^{-}, \PN)
  \]
  for some integer $n^{-} = n^{-}_{\sigma} \leq n$ as follows.
  Let $\Zo$ be an open subset consisting of smooth points of $Z$ and regard $f$ as $f|_{\Zo}$.
  We have the following composite homomorphism
  \begin{equation}\label{eq:def-Phi}
    \Phi = \Phi_f: f^*\sQ\spcheck \rightarrow f^*\sHom(\sHom(\sQ\spcheck , \sS\spcheck), \sS\spcheck)
    \rightarrow \sHom(T_{\Zo}, f^*\sS\spcheck),
  \end{equation}
  where
  the first homomorphism
  is induced from the dual of
  $\sS \otimes \sS\spcheck \rightarrow \sO$,
  and the second one
  is induced from the differential
  $df : T_{\Zo} \rightarrow f^*T_{\GN} = f^*\sHom(\sQ\spcheck , \sS\spcheck)$.
  In other words,
  $\Phi$ is the homomorphism corresponding to $df $ under the identification
  \begin{equation}\label{eq_identify_Phi}
    \Hom (f^*\sQ\spcheck, \sHom(T_{\Zo}, f^*\sS\spcheck)) \simeq \Hom (T_{Z^{\circ}},   f^*\sHom(\sQ\spcheck , \sS\spcheck)) .
  \end{equation}
  We define an integer $n^{-}$
  with $-1 \leq n^- \leq n$ by
  \begin{equation*}    n^{-} = \dim (\ker \Phi \otimes k(z)) - 1
  \end{equation*}
  for a general point $z \in Z$.
  In other word,  $n^{-} +1$ is the rank of the torsion free sheaf $\ker \Phi$.
  Then $\ker \Phi|_{\Zo}$ is a subbundle of
  $H^0(\PN, \sO(1))\spcheck \otimes \sO_{\Zo}$
  of rank $n^{-}+1$
  (replacing $\Zo \subset Z$ by an smaller open subset if necessary).
  By the universality of the Grassmann variety $ \Gr(n^{-}, \PN)$, 
  we have an induced morphism
  $\sigma: \Zo \rightarrow  \Gr(n^{-}, \PN)$
  and call it the \emph{shrinking map of $Z$ with respect to $f$}.
\end{defn}

For a rational map $f: X \dashrightarrow Y$ between varieties,
the \emph{rank} of $f$ (denoted by $\rk f$)
is defined to be the rank of the $\kk$-linear map
$d_xf$ for general $x \in X$,
where $d_xf : T_xX \rightarrow T_{f(x)}Y$
is the differential of $f$ between Zariski tangent spaces.
We note that $\rk f$ coincides with the rank of the ${K(X)}$-linear map
$\delta_{K(X)/K(f(X))}: \Omega_{K(f(X))/\kk} \otimes K(X) \rightarrow \Omega_{K(X)/\kk}$.

\begin{rem}\label{thm:rk0-shr-id}
  If the rank of $f: Z \dashrightarrow \GN$ is zero (i.e., $\delta_{K(Z)/K(f(Z))} = 0$),
  then
  the shrinking map $\sigma = \sigma_{Z,f}$ coincides with $f$.
  The reason is as follows:
  $\rk f  = 0$ implies that $T_Z \rightarrow f^*T_{\GN}$ is the zero map at the generic point of $Z$.
  Hence so is $\Phi$ in \ref{eq:def-Phi},
  that is, $\ker(\Phi) = f^*\sQ\spcheck$. This means that $\sigma = f$.
\end{rem}

\begin{rem}\label{rem_shrinking_map}\label{thm:shr-map-sep-map}
  Originally, the shrinking map is defined for a subvariety of $\GN$.
  If $f: Z \dashrightarrow \GN$ is separable onto its image $Y:= f(Z)$
  (i.e., $\rk f = \dim Y$),
  then $\sigma_{Z, f} = \sigma_{Y} \circ f$ holds, where
  $\sigma_{Y}$ is the shrinking map of $Y$ with respect to $Y \hookrightarrow \GN$.
  
  More generally, if $f: Z \dashrightarrow \GN$
  is decomposed as $Z \stackrel{g}{\dashrightarrow} Z' \stackrel{f'}{\dashrightarrow} \GN$
  and $g$ is dominant and separable,
  then $\sigma_{Z, f} = \sigma_{Z', f'} \circ g$ holds.
\end{rem}

Let us consider a subvariety $X' \subset \GN \times \PN$ contained in the universal family $\sU$ of $\GN$,
and denote by $\pr_1: X' \rightarrow \GN$ and $\pr_2 : X' \rightarrow \PN$
the first and second projections from $X'$ respectively.
We mainly study the case when the above $(Z, f)$ is $(X', \pr_1)$.

The rest of this section is devoted to the proof of the following theorem,
which gives a necessary and sufficient condition for $X' \subset \calu$ to be the graph of the Gauss map of some $X \subset \PN$.
We note that this theorem holds for any characteristic. 

\begin{thm}\label{thm:GaussImage}
  Let $X' \subset \GN \times \PN$ be an $n$-dimensional projective variety contained in $\sU$,
  such that $\pr_2: X' \rightarrow \PN$ is separable and generically finite.
  Let $\Phi = \Phi_{\pr_1}$ and $\sigma = \sigma_{X', \pr_1}$ be as in \autoref{thm:def-shr-Zf}.
  Then the following are equivalent.
  \begin{enumerate}[\normalfont(i)]
  \item\label{thm:GaussImage:graph}  $X'$ is the graph of the Gauss map $\gamma_X$ of some $n$-dimensional projective variety $X \subset \PN$.
  \item\label{thm:GaussImage:zero} The composition of 
    \[
    \sO_{\calu}(-1) |_{{X'}\spcirc} \hookrightarrow \pr_1^* \sQ^{\vee}
    \ \ \text{and}\ \ \,
    \Phi|_{{X'}\spcirc} : \pr_1^* \sQ^{\vee} \arw \sHom(T_{{X'}\spcirc}, \pr_1^* \sS^{\vee})
    \]
    is the zero map on a non-empty open subset ${X'}\spcirc \subset X'$.
  \item\label{thm:GaussImage:shr} The image of $(\sigma, \pr_2) : X' \dashrightarrow \Gr(n^{-}, \PN) \times \PN$
    is contained in 
    the universal family $\sU_{\Gr(n^{-}, \PN)} \subset \Gr(n^{-}, \PN) \times \PN$.
  \end{enumerate}

\end{thm}
Note that the condition~\ref{thm:GaussImage:graph} implies
$X = \pr_2(X')$; this is because,
if $X'$ is  the graph of $\gamma_X$, 
the composition of
$(\gamma_X,\id_X)$
and
$\pr_2 : X' \arw \PN$ is nothing but $\id_X$.
In other words, $\pr_2 : X' \arw X=\pr_2(X') \subset \PP^N$ is
the inverse of the birational map $(\gamma_X,\id_X) : X \dashrightarrow X' \subset \GN \times \PP^N$.

\begin{rem}\label{rem_Phi=0}
  If $\pr_1 : X' \arw \GN$ is of rank $0$, then
  the condition~\ref{thm:GaussImage:shr} of \autoref{thm:GaussImage} holds immediately
  because of \autoref{thm:rk0-shr-id}.
\end{rem}

\subsection{Notation}
\label{sec:notation}

We fix our notation.
Let $\sU$ be the universal family of $\GN$ as in \ref{eq:duf-univ}.
Let $\overline{\pr}_1 : \calu  \arw \GN$ and $\overline{\pr}_2 : \calu \arw  \PN$ be projections to the first and second factors respectively.
We recall that $\overline{\pr}_1 : \calu \arw \GN$ coincides with the projective bundle $\cP (\mathscr{Q}^{\vee}) \arw \GN$,
and the embedding 
\[
\cP (\mathscr{Q}^{\vee})=\sU \subset \GN \times \PN =\PP_{\! \! *} (H^0(\PN, \sO(1))\spcheck \otimes \sO_{\GN} )
\]
is induced by $\mathscr{Q}^{\vee} \subset H^0(\PN, \sO(1))\spcheck \otimes \sO_{\GN} $.
In particular,
we have the tautological invertible sheaf
$\sO_{\calu}(-1) = \overline{\pr}_2^* \mathscr{O}_{\PP^N}(-1) \subset \overline{\pr}_1^* \mathscr{Q}^{\vee}$ on $\calu$.

Take a basis
\[
Z^0, Z^1, \dots, Z^{N} \in H^0(\PN, \sO(1)),
\]
and regard $[Z^0: Z^1: \dots: Z^N]$ as the homogeneous coordinates on $\PN$.
Denote by
\[
Z_0, Z_1, \dots, Z_{N} \in H^0(\PN, \sO(1))\spcheck 
\]
the dual basis of $Z^0, Z^1,\dots, Z^N$.
Let $\Go \subset \GN$ be the open subset consisting of
$n$-planes not intersecting with the $(N-n-1)$-plane
$(Z^0 = Z^1 = \dots = Z^{n} = 0) \subset \PN$.
Then the sheaves $\sQ|_{\Go}$ and $\sS|_{\Go} \spcheck$
are free on $\Go$, and are equal to
$Q \otimes \sO_{\Go}$ and $S\spcheck \otimes \sO_{\Go}$ respectively
for the vector spaces
\[
Q = \bigoplus_{\Ri} \kk \cdot q^i \textand
S\spcheck = \bigoplus_{\Rj} \kk \cdot s_j,
\]
where $q^i $ is the image of $Z^i$ by $H^0(\PP^N,\sO(1)) \otimes \sO_{\GN} \arw \sQ$
and
$s_j$ is the image of $Z_j$ by $H^0(\PP^N,\sO(1))^{\vee} \otimes \sO_{\GN} \arw \sS^{\vee}$.
We denote by $\{q_i\}$ and $\{s^j\}$ the dual basis of $Q^{\vee}$ and $S:=(S^{\vee})^{\vee}$ respectively.
In this setting, we have a standard isomorphism
$\Go \simeq \Hom(Q\spcheck , S\spcheck) \simeq \A^{(n+1)(N-n)}$.
Hence we have $a_i^j \in \sO_{\Go}$ by
\[
(\aij(x))_{i,j}
= \sum_{\Rij} \aij(x) \cdot q^i \otimes s_j \in \Hom(Q\spcheck , S\spcheck)
\]
for each $x \in \Go$. 
In this notation,
$H^0(\PP^N,\sO(1)) \otimes \sO_{\Go} \arw \sQ |_{\Go} = Q \otimes \sO_{\Go}$ is described as
\[
Z^i \mapsto q^i \quad (\Ri), \qquad Z^j \mapsto \sum_{\Ri} q^i a_i^j \quad (\Rj).
\]
Hence
$\overline{\pr}_2: \sU \subset \GN \times \PN \rightarrow \PN$
maps a point
$\left((\aij(x)), [z^0: \dots: z^{n}] \right) \in \sU|_{\Go} =\PP_{\! \! *} (\sQ^{\vee} |_{\Go} )  = \Go \times \cP (Q\spcheck)$
to
\begin{equation}\label{eq:pr_2-U-to-PN}
  \left[z^0: \dots: z^{n}: \sum_{\Ri} z^i a_i^{n+1}(x): \dots: \sum_{\Ri} z^i a_i^{N}(x) \right] \in \PN.
\end{equation}

\begin{lem}\label{lem_loc_description_phi}
  Let $X' \subset \sU$ be a subvariety
  and let $\pr_1 =\overline{\pr}_1 |_{X'} : X' \arw \G(n,\PP^N)$ be the first projection.
  Assume $X' \cap \sU|_{\Go} \neq \emptyset$.
  Then the restriction of the composite homomorphism
  \begin{equation}\label{eq:OU-to-Omega-S}
    \mathscr{O}_{\calu}(-1) |_{X'} \hookrightarrow  \pr_1^* \sQ^{\vee} \stackrel{\Phi}{\longrightarrow} \Omega_{X'} \otimes \pr_1^* \sS^{\vee}
  \end{equation}
  on $X' \cap \sU|_{\Go}$ corresponds to the section
  \[
  \sum_{\Rj} \sum_{\Ri} z^i |_{X'} \cdot  \pr_1^*(d a_{i}^j) \otimes  s_j 
  \in H^0(X' \cap \sU|_{\Go} , \mathscr{O}_{\calu}(1) |_{X'} \otimes  \Omega_{X'} \otimes \pr_1^* \sS^{\vee})
  \]
  under the identification
  $\sHom(\mathscr{O}_{\calu}(-1) |_{X'},   \Omega_{X'} \otimes \pr_1^* \sS^{\vee})
  \simeq \mathscr{O}_{\calu}(1) |_{X'} \otimes  \Omega_{X'} \otimes \pr_1^* \sS^{\vee}$,
  where $\pr_1^*(d\aij)$ is the image of $d\aij$ under
  $\pr_1^*\Omega_{\GN} \rightarrow \Omega_{X'}$.
  
  In particular, the homomorphism \ref{eq:OU-to-Omega-S}
  is zero at the generic point of $X'$
  if and only if $\sum_{\Ri} z^i|_{X'} \cdot \pr_1^*(d a_{i}^j) =0$ in $\Omega_{K(X')/\kk}$ holds for any $ \Rj$.
\end{lem}

\begin{proof}
  By the definition of
  \[
  \Phi :  \pr_1^* \sQ^{\vee} \arw \pr_1^* ( \sQ^{\vee} \otimes  \sS \otimes \sS^{\vee}) =
  \pr_1^* ( \Omega_{\G(n,\PP^N)} \otimes \sS^{\vee})  \arw \Omega_{X'} \otimes \pr_1^* \sS^{\vee},
  \]
  $\Phi $ is locally described on ${X' \cap \sU|_{\Go}}$ by
  \[
  q_i \mapsto q_i \otimes \sum_{\Rj}  (s^j \otimes s_j) =\sum_{\Rj} d a_i^j \otimes s_j \mapsto \sum_{\Rj} \pr_1^*(d a_{i}^j) \otimes  s_j .
  \]
  Since $\mathscr{O}_{\calu}(-1) |_{X'} \hookrightarrow  \pr_1^* \sQ^{\vee}$ corresponds to the section
  \[
  \sum_{\Ri} z^i|_{X'} \cdot q_i
  \in H^0({X' \cap \sU|_{\Go}}, \sO_{\sU}(1)|_{X'} \otimes \pr_1^*\sQ\spcheck),
  \]
  this lemma follows.
\end{proof}

\subsection{Equality between two homomorphisms}
\label{sec:equality-between-two}

We have the following commutative diagram on $\calu$ with exact rows and columns:
\begin{equation}\label{eq:Euler-on-U}
  \xymatrix{    & 0 \ar[d] & 0 \ar[d]
    \\
    & \sO_{\sU}(-1) \ar[d] \ar@{=}[r] & \overline{\pr}_2^*(\sO_{\PN}(-1)) \ar[d]
    \\
    0 \ar[r] & \overline{\pr}_1^*\sQ\spcheck \ar[r] \ar[d]
    &  H^0(\PN, \sO(1))\spcheck \otimes \sO_{\sU} \ar[r] \ar[d] & \overline{\pr}_1^*\sS\spcheck \ar[r] \ar[d]^{\simeq} & 0
    \\
    0 \ar[r] & T_{\sU/\GN}(-1) \ar[r]    \ar[d]
    & \overline{\pr}_2^*T_{\PN}(-1) \ar[r]^(.42){\epsilon} \ar[d] & N_{\sU/\GN \times \PN} (-1) \ar[r] & 0 \makebox[0pt]{\,,}
    \\
    &0 & 0
  }\end{equation}
where the left column is the Euler sequence of $\overline{\pr}_1: \sU = \cP(\sQ\spcheck) \rightarrow \GN$
and the middle column is the pullback of the Euler sequence on $\PP^N$ under $\overline{\pr}_2$.
The middle row is the pullback of the dual of the universal sequence \ref{eq_univ_seq_U} under $\overline{\pr}_1$
and the bottom raw is obtained by tensoring $\mathscr{O}_{\calu}(-1)$ with the following natural sequence
\[
0 \arw T_{\calu/\GN} \arw T_{\GN \times \PP/ \GN} |_{\calu} \arw N_{\calu/ (\GN \times \PP)} \arw 0
\]
since $\overline{\pr}_2^*T_{\PN} = T_{\GN \times \PP/ \GN} |_{\calu} $.

\begin{lem}\label{lem_alpha_1=alpha_2}
  The following two homomorphisms $\alpha_1$ and $\alpha_2$ coincide:
  \begin{gather*}
    \alpha_1: T_{\sU}(-1) \xrightarrow{d (\overline{\pr}_1)(-1)} \overline{\pr}_1^* T_{\GN}(-1) = \overline{\pr}_1^* \sHom(\sQ \spcheck, \sS \spcheck)(-1) \rightarrow     \overline{\pr}_1^* \sS\spcheck,
    \\
    \alpha_2: T_{\sU}(-1) \xrightarrow{d (\overline{\pr}_2)(-1)} \overline{\pr}_2^*T_{\PN}(-1) \xrightarrow{\epsilon} \overline{\pr}_1^* \sS\spcheck,
  \end{gather*}
  where $d(\overline{\pr}_i)(-1)$ is the $(-1)$-twist of $d(\overline{\pr}_i)$,
  $\overline{\pr}_1^*\sHom(\sQ \spcheck, \sS \spcheck)(-1) \rightarrow \overline{\pr}_1^* \sS\spcheck$
  is induced from 
  $\mathscr{O}_{\calu}(-1) \subset \overline{\pr}_1^* \sQ^{\vee}$,
  and $\epsilon$ is in the bottom row of the diagram~\ref{eq:Euler-on-U}.
\end{lem}

\begin{proof}
  It is sufficient to show $(\alpha_1)_u = (\alpha_2)_u$ at each $u =(u_1,u_2) \in \calu |_{\Go} \subset \Go \times \PP^N$. 
  We set $Z^0=z^0=1$, and regard $Z^1,\ldots,Z^N$ and $z^1,\ldots,z^n$ as (non-homogeneous) coordinates on $\PP^N$ and $\PP_{\! \! *}(Q^{\vee})$ respectively. 
  
  First, we consider
  $(\alpha_2)_u : T_{\sU, u} \rightarrow T_{\PN,u_2} \rightarrow S\spcheck$.
  From \ref{eq:pr_2-U-to-PN}, the linear map
  $T_{\sU, u} \rightarrow T_{\PN, u_2}$ is described by
  \begin{align*}
    &\frac{\partial}{\partial z^k} \mapsto  \frac{\partial}{\partial Z^k} + \sum_{\Rj} a_k^{j}(u) \cdot \frac{\partial}{\partial Z^j} & (\Rk),
    \\
    &\frac{\partial}{\partial \aij} \mapsto  z^i \cdot \frac{\partial}{\partial Z^j} & (\Rij).
  \end{align*}
  On the other hand,
  $T_{\PN, u_2} \rightarrow S\spcheck$ is described by
  $\diff{}{Z^k} \mapsto - \sum_{\Rj} a_k^j s_j$ for $\Rk$ and $\diff{}{Z^j} \mapsto s_j$ for $\Rj$.
  Hence,
  \begin{align*}
    &(\alpha_2)_u\left( \frac{\partial}{\partial z^k} \right) =  0 & (\Rk),
    \\
    &(\alpha_2)_u\left( \frac{\partial}{\partial \aij} \right) =  z^i\cdot s_j & (\Rij).
  \end{align*}

  Next we consider
  $(\alpha_1)_{u}: T_{\sU,u} \rightarrow T_{\GN, u_1} \rightarrow S\spcheck$.
  The linear map
  $T_{\sU,u} \rightarrow T_{\GN, u_1} = \Hom(Q\spcheck, S\spcheck)$
  is described by
  $\diff{}{z^k} \mapsto 0$
  $(\Rk)$
  and
  $\diff{}{\aij} \mapsto q^i \otimes s_j$
  $(\Rij)$.
  On the other hand,
  $T_{\GN, u_1} = \Hom(Q\spcheck, S\spcheck) \rightarrow S\spcheck$
  is given by $q^i \otimes s_j \mapsto z^i \cdot s_j$.
  Therefore $(\alpha_1)_u(\diff{}{z^k}) = 0$ and $(\alpha_1)_u(\diff{}{\aij}) = z^i \cdot s_j$,
  which means that $(\alpha_1)_u = (\alpha_2)_u$.
\end{proof}

\subsection{Characterization of graphs of Gauss maps}

\begin{lem}\label{thm:graph-equiv}
  Let $X' \subset \GN \times \PP^N$ be an $n$-dimensional projective variety
  and set $X = \pr_2(X') \subset \PN$.
  Then the following are equivalent.
  \begin{enumerate}[\quad \normalfont(1)]
  \item $X'$ is the graph of the Gauss map of some $n$-dimensional projective variety in $\PN$.
  \item $\dim X = n$ and $X'$ is the graph of the Gauss map $\gamma_X$ of $X \subset \PN$.
  \item $\dim X = n$ and $\gamma_X \circ \pr_2 =\pr_1$.
  \end{enumerate}
\end{lem}
\begin{proof}
  (2) $ \Rightarrow $ (1) holds immediately.
  (1) $ \Rightarrow $ (2) holds as we already noted after the statement of \autoref{thm:GaussImage}.

  (2) $ \Leftrightarrow $ (3):
  If $X'$ is  the graph of $\gamma_X$, then $(\gamma_X,\id_X): X \dashrightarrow X'$ is birational and its inverse is $\pr_2$.
  Since $\pr_1 \circ (\gamma_X,\id_X) =\gamma_X$,
  we have $\pr_1   =\gamma_X \circ (\gamma_X,\id_X)^{-1}=\gamma_X \circ \pr_2$.

  Conversely, assume that $\dim X=n$ and $\pr_1   =\gamma_X \circ \pr_2$.
  Let $x' =([L],x) $ be a general point of $X' \subset \GN \times \PP^N$.
  Then we have $[ L] =\pr_1 (x')= \gamma_X \circ \pr_2 (x')=\gamma_X (x)$ in $\GN$.
  Hence a general point of $X'$ is written as $(\gamma_X(x), x)$ for some $x \in X$,
  which means that $X'$ coincides with the graph of $\gamma_X$.
\end{proof}

\begin{proof}[Proof of \autoref{thm:GaussImage}]

  \ref{thm:GaussImage:graph} $\Leftrightarrow$ \ref{thm:GaussImage:zero}.
  It is sufficient to check the equivalence between \ref{thm:GaussImage:zero} of the theorem
  and (3) of \autoref{thm:graph-equiv}.
  We set $\sO_{X'}(-1) := \sO_{\calu}(-1)|_{{X'}}$.

  Let 
  $ 0 \rightarrow \sQ\spcheck \stackrel{\jmath}{\rightarrow} H^0(\PN, \sO(1))\spcheck \otimes \sO_{\GN} \stackrel{\varpi}{\rightarrow} \sS\spcheck \rightarrow 0$
  be the dual of the universal exact sequence \ref{eq_univ_seq_U}.
  By the definition of the Gauss map $\gamma=\gamma_X$,
  we have the following diagram on the smooth locus $X_{sm}$ of $X$.
  \begin{equation}\label{eq:QLS-TTN}
    \begin{aligned}
      \xymatrix{    0 \ar[r] & \gamma^*\sQ\spcheck \ar[r]^(.3){\gamma^* \jmath} \ar[d]_{\theta}
        &  H^0(\PN, \sO(1))\spcheck \otimes \sO_{X_{sm}} \ar[r]^(.7){\gamma^* \varpi} \ar[d] & \gamma^*\sS\spcheck \ar[r] \ar[d]^{\simeq} & 0
        \\
        0 \ar[r] & T_{X_{sm}}(-1) \ar[r]
        & T_{\PN}(-1) |_{X_{sm}} \ar[r]  & N_{X_{sm}/\PP^N} (-1) \ar[r] & 0 .
      }      \end{aligned}
  \end{equation}

  Since $\pr_2 : X' \arw  X = \pr_2 (X') \subset  \PP^N$ is separable and generically finite,
  we can take an open subset $X'^{\circ} \subset X'$ such that $X'^{\circ}$ is smooth and $\pr_2 |_{X'^{\circ}}$  is \'{e}tale.
  Now we regard $\pr_2$ as $\pr_2|_{{X'}\spcirc}$.
  Since $d (\pr_2) : T_{X'^{\circ}} \arw  \pr_2^* T_{X} $ is an isomorphism,
  we have
  \begin{equation}\label{diag_gamma-p_2}
    \begin{aligned}
      \xymatrix{    0 \ar[r] & (\gamma \circ \pr_2)^*\sQ\spcheck \ar[r]^(.4){(\gamma \circ \pr_2)^* \jmath} \ar@{->>}[d]
        &  H^0(\PN, \sO(1))\spcheck \otimes \sO_{{X'}\spcirc} \ar[r]^(.57){(\gamma \circ \pr_2)^* \varpi} \ar[d] & (\gamma \circ \pr_2)^*\sS\spcheck \ar[r] \ar[d]^{\simeq} & 0
        \\
        0 \ar[r] & T_{X'^{\circ}}(-1) \ar[r]^(.44){d(\pr_2)(-1)}
        & \pr_2^* T_{\PN}(-1) |_{{X'}\spcirc} \ar[r]  & \pr_2^* N_{X_{sm}/\PP^N} (-1) \ar[r] & 0
      }      \end{aligned}
  \end{equation}
  by pulling back the diagram~\ref{eq:QLS-TTN} by $\pr_2 : X'^{\circ} \arw X$.

  Since $ \sQ\spcheck \subset H^0(\PN, \sO(1))\spcheck \otimes \sO_{\GN}$
  is the universal subbundle,
  $\gamma \circ \pr_2 = \pr_1$ holds if and only if
  \[
  (\gamma \circ \pr_2)^*\sQ\spcheck =  \pr_1^* \sQ\spcheck
  \]
  holds as subsheaves of $H^0(\PN, \sO(1))\spcheck \otimes \sO_{{X'}\spcirc}$.
  Since $(\gamma \circ \pr_2)^*\sQ\spcheck$ and $ \pr_1^* \sQ\spcheck$ have the same rank,
  $(\gamma \circ \pr_2)^*\sQ\spcheck= \pr_1^* \sQ\spcheck$ holds if and only if
  $(\gamma \circ \pr_2)^*\sQ\spcheck$ is contained in $ \pr_1^* \sQ\spcheck$.
  Since $ \pr_1^* \sQ\spcheck$ is the kernel of $H^0(\PN, \sO(1))\spcheck \otimes \sO_{{X'}\spcirc} \stackrel{\pr_1^* \varpi}{\longrightarrow} \pr_1^* \sS\spcheck$,
  the inclusion $(\gamma \circ \pr_2)^*\sQ\spcheck \subset \pr_1^* \sQ\spcheck$ holds
  if and only if
  \begin{align}\label{eq_2q-s1}
    (\gamma \circ \pr_2)^*\sQ\spcheck \stackrel{(\gamma \circ \pr_2)^* \jmath}{\longrightarrow} 
    H^0(\PN, \sO(1))\spcheck \otimes \sO_{{X'}\spcirc} \stackrel{\pr_1^* \varpi}{\longrightarrow} \pr_1^* \sS\spcheck
  \end{align}
  is the zero map.
  We have the following commutative diagram
  \[
  \xymatrix{      (\gamma \circ \pr_2)^*\sQ\spcheck \ar[r]^(.4){(\gamma \circ \pr_2)^* \jmath} \ar@{->>}[d]
    &  H^0(\PN, \sO(1))\spcheck \otimes \sO_{{X'}\spcirc} \ar[r]^(.65){\pr_1^* \varpi} \ar[d] & \pr_1^*\sS\spcheck  \ar[d]^{\simeq} 
    \\
    T_{X'^{\circ}}(-1) \ar[r]^(.45){d(\pr_2) (-1)} \ar@/_1.5pc/[rr]_{\beta}& \pr_2^* T_{\PN}(-1)  \ar[r]^(.4){\ep}  & N_{\sU/\GN \times \PN} (-1) |_{{X'}\spcirc},
  }\]
  where the commutativity of the left and right squares follows from \ref{diag_gamma-p_2} and \ref{eq:Euler-on-U} respectively.
  Since $(\gamma \circ \pr_2)^*\sQ\spcheck  \arw T_{X'^{\circ}}(-1) $ is surjective,
  \ref{eq_2q-s1} is the zero map if and only if so is the bottom row
  $\beta: T_{{X'}\spcirc}(-1) \rightarrow N_{\sU/\GN \times \PN} (-1)|_{{X'}\spcirc} \simeq \pr_1^* \sS\spcheck$.
  Since $d (\pr_2) (-1)$ is decomposed as
  \[
  T_{{X'}\spcirc}(-1) \hookrightarrow T_{\calu}(-1)|_{{X'}\spcirc} \stackrel{d (\overline{\pr}_2) (-1)}{\longrightarrow} \overline{\pr}_2^* T_{\PN}(-1)|_{{X'}\spcirc} = \pr_2^* T_{\PN}(-1) ,
  \]
  $\beta$ is the composition of $T_{{X'}\spcirc}(-1) \hookrightarrow T_{\calu}(-1) |_{{X'}\spcirc}$ and $\alpha_2 |_{{X'}\spcirc}$.
  By \autoref{lem_alpha_1=alpha_2}, we have $\alpha_1 =\alpha_2$.
  By the definition of $\alpha_1$, the homomorphism
  $\beta \in \Hom(T_{{X'}\spcirc}(-1), \pr_1^* \sS\spcheck)$ is obtained with the following commutative diagram
  \[
  \xymatrix{  T_{{X'}\spcirc}(-1) \ar[r]^(.38){ d(\pr_1) (-1)} \ar@{^(->}[d] & {\pr}_1^* \sHom(\sQ \spcheck, \sS \spcheck)(-1) \ar[dr] \ar@{^(->}[d]
    \\
    T_{{X'}\spcirc} \otimes \pr_1^*\sQ\spcheck \ar[r] & {\pr}_1^* \sHom(\sQ \spcheck, \sS \spcheck) \otimes \pr_1^*\sQ\spcheck \ar[r] & {\pr}_1^* \sS\spcheck,
  }\]
  where we recall that $T_{\GN} = \sHom(\sQ \spcheck, \sS \spcheck)$.
  By construction,
  the bottom map
  $T_{{X'}\spcirc} \otimes \pr_1^*\sQ\spcheck \arw  {\pr}_1^* \sS\spcheck$ corresponds to $d (\pr_1) : T_{{X'}\spcirc} \arw\pr_1^* T_{\GN}$
  under the identification
  \[
  \Hom (T_{{X'}\spcirc} \otimes \pr_1^*\sQ\spcheck,  {\pr}_1^* \sS\spcheck) \simeq \Hom(T_{{X'}\spcirc} , \pr_1^* \sHom(\sQ \spcheck, \sS \spcheck)).
  \]
  Since $\Phi = \Phi_{\pr_1}$ also corresponds to $d (\pr_1) $ by \ref{eq_identify_Phi},
  $\beta$ corresponds to the composite homomorphism
  \begin{align}\label{eq_O(-1)_to}
    \sO_{{X'}\spcirc}(-1) \hookrightarrow \pr_1^* \sQ^{\vee} \stackrel{\Phi}{\longrightarrow} \sHom(T_{{X'}\spcirc}, \pr_1^* \sS\spcheck)
  \end{align}
  under the identification
  \[
  \Hom (T_{{X'}\spcirc}(-1) ,\pr_1^*\sS\spcheck )\simeq \Hom(\sO_{{X'}\spcirc}(-1), \sHom(T_{{X'}\spcirc}, \pr_1^* \sS\spcheck)).
  \]
  Hence $(\gamma \circ \pr_2)^*\sQ\spcheck= \pr_1^* \sQ\spcheck$ holds if and only if
  \ref{eq_O(-1)_to} vanishes, and \ref{thm:GaussImage:graph} $\Leftrightarrow$ \ref{thm:GaussImage:zero} is proved.
  \\

  \noindent
  \ref{thm:GaussImage:zero} $\Leftrightarrow$ \ref{thm:GaussImage:shr}.
  Let
  \[
  0 \arw \sS^- \arw H^0(\PP^N,\sO(1)) \otimes \sO_{\G(n^-,\PP^N)} \arw \sQ^- \arw 0
  \]
  be the universal exact sequence on $\G(n^-,\PP^N)$,
  and let $\widetilde{\pr}_1 , \widetilde{\pr}_2$ be the projections from $\G(n^-,\PP^N) \times \PP^N $ to the first and second factors respectively.
  We note that the universal family $\calu_{\G(n^-,\PP^N)} \subset \G(n^-,\PP^N) \times \PP^N$ is nothing but the locus
  where the composite homomorphism
  \[
  \widetilde{\pr}_2^* \sO_{\PP^N}(-1) \hookrightarrow H^0(\PP^N,\sO(1))^{\vee} \otimes \sO_{\G(n^-,\PP^N) \times \PP^N} \arw \widetilde{\pr}_1 (\sS^{-})^{\vee}
  \]
  vanishes.
  Let $\sigma = \sigma_{X', \pr_1}$ be the shrinking map of $X'$ and take an open subset
  ${X'}\spcirc$ consisting of smooth points at which $\sigma$ is defined.
  Then $(\sigma, \pr_2)({X'}\spcirc) \subset \calu_{\G(n^-,\PP^N)}$ if and only if
  \[
  \sO_{{X'}\spcirc}(-1) = \pr_2^* \sO_{\PP^N}(-1) \arw H^0(\PP^N,\sO(1))^{\vee} \otimes \sO_{{X'}\spcirc} \arw  \sigma^* ( \sS^{-})^{\vee}
  \]
  vanishes.
  Since the kernel of $H^0(\PP^N,\sO(1))^{\vee} \otimes \sO_{{X'}\spcirc} \arw \sigma^* ( \sS^{-})^{\vee}$ is $\sigma^* ( \sQ^{-})^{\vee} $,
  the inclusion $(\sigma, \pr_2)({X'}\spcirc) \subset \calu_{\G(n^-,\PP^N)}$ holds if and only if
  $\sO_{{X'}\spcirc}(-1)$ is contained in $\sigma^* ( \sQ^{-})^{\vee} $ as a subsheaf of $H^0(\PP^N,\sO(1))^{\vee} \otimes \sO_{{X'}\spcirc}$.

  Since $\sigma^* ( \sQ^{-})^{\vee} = \ker\Phi$ by the definition of $\sigma$,
  $\sO_{{X'}\spcirc}(-1)$ is contained in $\sigma^* ( \sQ^{-})^{\vee} $ if and only if
  $\sO_{{X'}\spcirc}(-1) \hookrightarrow  \pr^*_1 \sQ^{\vee} \stackrel{\Phi}{\arw} \sHom(T_{{X'}\spcirc}, \pr_1^* \sS\spcheck)$
  vanishes.
  Thus \ref{thm:GaussImage:zero} $\Leftrightarrow$ \ref{thm:GaussImage:shr} holds.
\end{proof}

\section{Field extensions and Gauss maps}\label{sec_filed_extension}

In this section, we prove Theorems \ref{thm:anyY}, \ref{thm_construction_for_insep}
by using \autoref{thm:GaussImage},
We mainly consider the case $p = \chara \kk > 0$.

\begin{defn}\label{def_sep_tr_basis}
  Let $L$ be a field over $\kk$ and set $n=\trdeg_{\kk} L$.
  A set $\{x_1,\ldots,x_n \} $ of $n$ elements in $L$ is called a \emph{separating transcendence basis} of $L /\kk$
  if $\{dx_1,\ldots,dx_n\}$ is a basis of $L$-vector space $\Omega_{L/\kk}$.
  In other words,
  $\{x_1,\ldots,x_n \} $ is a separating transcendence basis of $L /\kk$ if and only if
  $L/\kk(x_1,\dots, x_n)$ is a finite separable extension.
\end{defn}

First, we show a lemma about the number of generators of field extensions.

\begin{lem}\label{lem_number_of_generators}
  Let $L / K$ be a field extension over $\kk$.
  Set $n=\trdeg_{\kk} L$ and $m=\trdeg_{\kk} (K) $.
  Assume $m \geq 1$. Then
  there exists a separating transcendence basis $x_1,\ldots,x_n$ of $L / \kk$
  such that $L =K(x_1,\ldots,x_n)$.
  More strongly,
  we can take $x_i \in K$ for $i \leq \min \{ \rk_L(\delta_{L/K}) , m-1\}$.
\end{lem}

\begin{proof}
  Set $r= \rk_L(\delta_{L/K}) \leq m$.
  By the definition of $\delta_{L/K}$,
  there exist $x_{1}, \ldots , x_{r} \in K$ and $x_{r+1},\ldots, x_n \in L$
  such that $x_{1},\ldots, x_{n}$ form a separating transcendence basis of $L / \kk$.
  By renumbering the indices of $x_{r+1},\ldots, x_{n}$,
  we may assume that 
  $x_{m+1},\ldots, x_{n}$ form a transcendence basis of $L/K$.
  Then $L/K'$ is a finite extension for $K' := K(x_{m+1},\ldots,x_{n}) \subset L$.
  Let $K'_{sep} \subset L$ be the separable closure of $ K'$ in $L$.
  Since $L / K'_{sep}$ is purely inseparable,
  $K'_{sep} \supset L^{p^e} $ holds for $e \gg 0$.
  Since $x_{1},\ldots, x_{n} $ form a separating transcendence basis of $L / \kk$,
  we have $L=L^{p^e} (x_{1},\ldots, x_{n}) $ for any $e \geq 0$
  (this is because, $L/L^{p^e} (x_{1},\ldots, x_{n})$ is separable, and
  the separable degree $[L:L^{p^e} (x_{1},\ldots, x_{n})]_s$ is $1$ since so is
  $[L:L^{p^e}]_s$).
  Hence it holds that
  \[
  L=K'_{sep}(x_{1},\ldots, x_{n}) = K'_{sep}(x_{1},\ldots, x_{m}),
  \]
  where the second equality follows from $x_{m+1}, \ldots,x_n \in K' \subset K'_{sep}$.
  On the other hand,
  $K'_{sep} = K'(w)$ holds for some $w \in K'_{sep} $ since $K'_{sep} / K'$ is separable and finite.
  Applying \cite[Theorem 14]{Jacobson} to $K'(w, x_m) / K'$,
  we can take $\tilde{x}_m \in K'(w, x_m) $ such that $K'(w, x_m) =K'(\tilde{x}_m)$
  (we use the assumption $m \geq 1$ here).
  In fact,
  we can take $\tilde{x}_m = w + t x_m$ for general $t \in \kk$ by the proof of \cite[Theorem 14]{Jacobson}.
  Then
  \[
  d x_1,\ldots, d x_{m-1}, d \tilde{x}_m = d w + t d x_m ,d x_{m+1},\ldots,d x_n
  \]
  form a basis of $\Omega_{L / \kk}$
  since $t \in \kk$ is general.
  Hence $x_1, \ldots,x_{m-1}, \tilde{x}_m, x_{m+1},\ldots, x_{n}$ form a separating transcendence basis of $L / \kk$
  and it holds that
  \begin{align*}
    L= K'_{sep}(x_1,\ldots, x_{m}) &= K'(w) (x_1,\ldots, x_{m}) \\
    &=K' (x_1,\ldots, x_{m-1},\tilde{x}_m) \\
    &= K(x_1, \ldots,x_{m-1}, \tilde{x}_m, x_{m+1},\ldots, x_{n}).
  \end{align*}
  The last statement of this lemma follows from $x_1,\ldots,x_r \in K$.
\end{proof}

We can also show a geometric version of \autoref{lem_number_of_generators},
which we will use to prove \autoref{thm:family} in \autoref{sec_fibers}.

\begin{lem}\label{lem_number_of_generators_geom}
  Let $X \subset \PP^N$ be an $n$-dimensional projective variety over $\kk$ and
  let $f : X \dashrightarrow Y $ be a dominant rational map to a projective variety $Y$ with $m=\dim Y \geq 1$.
  Set $r= \rank f$ and  take an integer $M$ with $ \max \{n-m+1, n-r\} \leq M \leq N$.
  For a general linear projection $\pi : \PP^N \dashrightarrow \PP^{M}$,
  $(f,\pi |_X) : X \dashrightarrow Y \times \PP^{M}$ is birational onto the image.
\end{lem}

\begin{proof}
  Since $M \geq n-r = n-\rank f$,
  $(f,\pi |_X) : X \dashrightarrow Y \times \PP^{M}$ is separable and generically finite for general $\pi$.
  Hence it suffices to show that $(f,\pi |_X) $ is generically bijective.

  Let $x \in X$ be a general point and set $(y,z)=(f(x),\pi(x)) \in Y \times \PP^{M}$.
  We show that $(f,\pi |_X)^{-1}(y,z) =\{x\}$.

  Let $f^{-1}(y)=\bigcup F_i$ be the decomposition into the irreducible components with the reduced structures.
  Take a component $F_0$ such that $x \in F_0$.
  Since $y$ is general,
  $\dim F_i = n-m < M$ holds for each $i$.
  Hence $\pi |_{F_i} : F_i \dashrightarrow \pi(F_i) \subset \PP^{M}$ is birational because $\pi$ is a general projection.
  Furthermore,
  $\overline{\pi(F_i)} \neq \overline{\pi(F_j)} $ for $i \neq j$ by the generality of $\pi$.
  Hence $z=\pi(x)$ is not contained in $\pi(F_i)$ for any $i \neq 0$
  since $x \in f^{-1}(y)$ is general.

  Take $x' \in (f,\pi |_X)^{-1}(y,z) $.
  Since $z=\pi(x')$ is not contained in $\pi(F_i)$ for any $i \neq 0$,
  we have $x' \in F_0$.
  Since $\pi |_{F_0} : F_0 \dashrightarrow \pi(F_0) \subset \PP^{M}$ is birational and $x \in F_0$ is general,
  $\pi(x')=y=\pi(x)$ means $x'=x$ and this lemma holds.
\end{proof}

Now we can prove Theorems \ref{thm:anyY}.

\begin{proof}[Proof of \autoref{thm:anyY}]
  Let $\pi_1 : \Spec L \rightarrow Y$ be the morphism induced by the field extension $L /K(Y)$.
  By \autoref{lem_number_of_generators},
  we can take a separating transcendence basis $x_1,\ldots,x_n \in L$ of $L / \kk$ such that $L=K(Y)(x_1,\ldots,x_n)$.
  Let $\pi_2 : \Spec L \arw \A^n$ be the morphism induced by $\kk[x_1,\ldots,x_n] \subset L$. 
  Let $\Go$ be as in \autoref{sec:notation}.
  By choosing $Z^0,\ldots,Z^N \in H^0(\PP^N,\sO(1))$ generally,
  we may assume that $ Y \cap \G^{\circ} \neq \emptyset$.
  We consider an inclusion
  $(Y \cap \G^{\circ}) \times \A^n \hookrightarrow \G^{\circ} \times  \cP(Q^{\vee})$
  given by the embedding $\A^n \hookrightarrow \cP(Q^{\vee}) : (x_1,\ldots,x_n) \mapsto [1:x_1 : \cdots : x_n]$.
  We recall $\Go \times \cP (Q\spcheck) = \PP_{\! \! *} (\sQ^{\vee} |_{\Go} )  = \sU|_{\Go} \subset \sU$.
  We define $X' \subset \calu$
  to be the closure of the image of the composite morphism
  \[
  (\pi_1,\pi_2) : \Spec L \arw (Y \cap \G^{\circ}) \times \A^n
  \hookrightarrow \calu.
  \]

  Since $L=K(Y)(x_1,\ldots,x_n)$,
  the morphism $(\pi_1,\pi_2) : \Spec L \arw X'$ is birational
  and the extension $K(X') / K(Y)$ induced by $\pr_1 : X' \twoheadrightarrow Y \subset \GN$ coincides with $L/K(Y)$.
  Thus it suffices to show that $X'$ satisfies conditions in \autoref{thm:GaussImage}.

  Since $x_1,\ldots,x_n$ form a separating transcendence basis of $L / \kk$,
  $\pr_2 : X' \arw \PP^N$ is generically finite and separable by the local description \ref{eq:pr_2-U-to-PN}.
  Since $\pr_1 : X' \arw \GN$ is of rank $0$ (i.e., $\delta_{K(X')/K(Y)} = \delta_{L/K(Y)} = 0$),
  \ref{thm:GaussImage:shr} in \autoref{thm:GaussImage} is satisfied by \autoref{rem_Phi=0}.
  Hence \autoref{thm:anyY} holds.
\end{proof}

\begin{ex}\label{ex_Y_with_r>1}
  In this example,
  we see that we cannot replace the condition $\delta_{L/K(Y)} = 0$ by the inseparability of $L/K(Y)$ in \autoref{thm:anyY}.

  Assume $p >0$.
  Let $Y \subset \G(2, \PP^3)$ be the surface given as the closure of the image of
  an embedding $\A^2  \hookrightarrow \G(2,\PP^3)$, which maps $(x,y) \in \A^2$
  to the $2$-plane in $\PP^3$ spanned by the $3$ points corresponding to the row vectors of
  \begin{align*}
    \begin{bmatrix}
      1 & 0 &0  & 1
      \\
      0 & 1 & 0 & x
      \\
      0 & 0 & 1 & y
    \end{bmatrix}.
  \end{align*}
  Set $L=\kk (x,y^{1/p}) $,
  which contains $K(Y)= \kk(x,y)$.
  We show that there is \emph{no} surface $X \subset \PP^3$
  with $K(X)=L$ such that $\overline{\gamma_X(X)} =Y$
  and the extension $K(X)/ K(Y)$ induced by $\gamma_X$ coincides with given $L / K(Y)$.

  Suppose there exists such $X \subset \PP^3$ and let
  $X' \subset \calu \subset \G(2,\PP^3) \times \PP^3$ be the graph of $\gamma_X$.
  We use the notation in \autoref{sec:notation}
  with the setting
  \[
  (n, N) = (2, 3) \ \text{ and }\
  (a_0^3, a_1^3, a_2^3) = (1, x, y).
  \]
  From \autoref{thm:GaussImage},
  $\mathscr{O}_{\calu}(-1) |_{X'} \hookrightarrow  \pr_1^* \sQ^{\vee} \arw \Omega_{X'} \otimes \pr_1^* \sS^{\vee}$
  is zero at a general point of $X'$.
  Since $\pr_1(X') \cap  \G^{\circ} = Y \cap \G^{\circ} \neq \emptyset$,
  so is $X' \cap \calu |_{\G^{\circ}}$.
  Hence
  \[
  0= z^0 d a_0^3 + z^1 d a_1^3 + z^2 d a_2^3 = z^1 d x
  \] holds in $\Omega_{K(X')/\kk}=\Omega_{L/\kk}$ by \autoref{lem_loc_description_phi}.
  Since $d x \neq 0 $ in $\Omega_{L/\kk}$,
  it holds that $z^1 =0$ on $X'$.
  By \ref{eq:pr_2-U-to-PN},
  this means that $X=\pr_2(X')$ is contained in the hyperplane $(Z^1=0) \subset \PP^3$.
  Thus $X=(Z^1=0) $ holds and $\gamma_X $ is a constant map,
  which is a contradiction.
\end{ex}

In the proof of \autoref{thm:anyY},
the condition (iii) in \autoref{thm:GaussImage} immediately follows from the assumption that
$\delta_{L/K(Y)} = 0$. In the case of \autoref{thm_construction_for_insep},
we choose the generators of $L/K$
and an embedding $\Spec K \hookrightarrow \G(n,\PP^{n+1})$ carefully
and apply \autoref{thm:GaussImage}.

\begin{proof}[Proof of \autoref{thm_construction_for_insep}]
  In order to prove this theorem, it suffices to show the case $N=n+1$ (see \autoref{thm_construction_for_insep:rem} below).
  Thus we will construct a hypersurface $X \subset \PP^{n+1}$ whose Gauss map induces
  the given extension $L/K$.

  Set $m=\trdeg_{\kk} (K)$ and $r = \rk_L(\delta_{L/K})$.
  Since $L/K$ is inseparable,
  $0 \leq r \leq m-1$ holds.
  By \autoref{lem_number_of_generators},
  there exists a separating transcendence basis $x_{1},\ldots, x_{n}$ of $L / \kk$
  with $x_1,\ldots,x_{r} \in K$ and $L=K(x_{1},\ldots,x_n)=K(x_{r+1},\ldots,x_n)$.

  Since $d x_1,\ldots,d x_r \in \Omega_{K/\kk}$ are linearly independent,
  there exist $y_{r+1},\ldots,y_m \in K$ such that
  $x_1,\ldots,x_r,y_{r+1}, \ldots, y_m$ form a separating transcendence basis of $K / \kk$.
  Since $K / \kk(x_1,\ldots,x_r,y_{r+1}, \ldots, y_m)$ is separable and finite,
  we can take $ a \in K$ such that $K = \kk(x_1,\ldots,x_r,y_{r+1}, \ldots, y_m,a)$.

  Since $d x_1,\ldots , d x_r$ form a basis of the image of $\delta_{L/K}: \Omega_{K/\kk} \otimes_{K} L \arw \Omega_{L/\kk}$ as an $L$-vector space,
  \begin{align}\label{eq_b_i}
    d a + x_{r+1} d y_{r+1} + \cdots + x_m d y_m = \sum_{i=1}^r b_i d x_i
  \end{align}
  holds for some $b_i \in L$ as elements in $\Omega_{L/\kk}$.

  Now we construct an embedding $\Spec L \rightarrow \sU \subset \Gr(n, \PP^{n+1}) \times \PP^{n+1}$ such that the closure of its image
  $X' \subset \sU$
  satisfies the condition
  $\sum_{\Ri} z^i \pr_1^*(d a_{i}^{n+1}) =0$   in \autoref{lem_loc_description_phi}.
  The following claim plays a key role.

  \begin{claim}\label{clm_choice_of_f}
    Assume $p \geq 3$ or $r$ is even.
    Then, there exists $f \in \kk(x_1,\ldots,x_r)$ such that
    $f_{x_1}-b_1, \ldots, f_{x_r} -b_r, x_{r+1},\ldots,x_n$ form a separating transcendence basis of $L/\kk$ for $f_{x_i } := \partial f / \partial x_i$.
  \end{claim}

  \begin{proof}[Proof of \autoref{clm_choice_of_f}]
    First, we consider the case $p \geq 3$.
    Set $f = t(x_1^2 + \cdots +x_r^2)$ for general $t \in \kk$.
    Then $f_{x_i}= 2t x_i$ for $1 \leq i \leq r$ and there exists $g(t) \in L$ such that
    \begin{equation}\label{eq:df-b}
      d (f_{x_1} -b_1) \wedge \cdots \wedge d (f_{x_r} -b_r) \wedge d x_{r+1} \wedge \cdots \wedge d x_n
      =g(t) d x_1 \wedge \cdots \wedge d x_n
    \end{equation}
    holds as elements in $ \bigwedge^n \Omega_{L/\kk} $.
    Since $g(t)$ is written as a nonzero polynomial of $t$ with degree $r$ (the coefficient of $t^r$ is $2^r \neq 0$),
    $g(t) \neq 0$ for general $t \in \kk$.
    Thus
    $f_{x_1}-b_1, \ldots, f_{x_r} -b_r, x_{r+1},\ldots,x_m$ form a separating transcendence basis of $L/\kk$.

    Next, we consider the case $r$ is even. Set $s=r/2$ and
    $f = t \cdot \sum_{1 \leq l \leq s} x_{2l-1}x_{2l}$
    for general $t \in \kk$.
    Then
    $f_{x_{2l-1}}= t x_{2l}$,
    $f_{x_{2l}}= t x_{2l-1}$, and
    we take $g(t) \in L$ as in \ref{eq:df-b}. Then
    $g(t)$ is written as a nonzero polynomial of $t$ with degree $r$ (the coefficient of $t^r$ is $(-1)^s \neq 0$ in this case). Hence
    $g(t) \neq 0$ for general $t \in \kk$.
    Thus
    $f_{x_1}-b_1, \ldots, f_{x_r} -b_r, x_{r+1},\ldots,x_n$ form a separating transcendence basis of $L/\kk$.
  \end{proof}

  Take $f \in \kk(x_1,\ldots,x_r)$ as in \autoref{clm_choice_of_f}
  and let $z^1,\ldots,z^n \in L$ be
  \[
  f_{x_1}-b_1, \ldots, f_{x_r} -b_r, x_{r+1},\ldots,x_n
  \]
  respectively.
  Then it holds that
  \begin{equation}\label{eq_linear_condition}
    \begin{aligned}
      d(a-f) &+ \sum_{i=1}^r z^i d x_i  + \sum_{i=r+1}^{m} z^{i} d y_{i} \\
      &= d(a -f) + \sum_{i=1}^r (f_{x_i} -b_i)d x_i +\sum_{i=r+1}^{m} x_{i} d y_{i} \\
      &=da -df + df - \sum_{i=1}^r b_i d x_i  +\sum_{i=r+1}^{m} x_{i} d y_{i} =0,
    \end{aligned}
  \end{equation}
  where the second equality follows from $df = \sum_{i=1}^r f_{x_i} d x_i$, and the last equality follows from \ref{eq_b_i}.

  Now we use the notation of \autoref{sec:notation}.
  Let $Y \subset \G(n,\PP^{n+1})$ be the closure of the morphism
  \[
  \mu : \Spec K \arw  \Go \subset \G(n,\PP^{n+1})
  \]
  defined by
  \begin{align*}\label{a_i^m+1}
    a_i^{n+1} =  \left\{ 
      \begin{array}{cl} 
        a-f & \text{if } i=0 ,\\ 
        x_i & \text{if } 1 \leq i \leq r ,\\
        y_i & \text{if } r+1 \leq i \leq m, \\
        0 & \text{if } m+1 \leq i \leq n.
      \end{array} \right.
  \end{align*}
  Since $\kk(a-f,x_1,\ldots, x_r,y_{r+1},\ldots,y_m )=K$,
  $\mu : \Spec K \arw Y \subset \G(n,\PP^{n+1})$ is birational.

  Let $\nu_1 :\Spec L \arw Y$ be the composition of $\Spec L \arw \Spec K$ and $\mu$,
  and let $\nu_2 : \Spec L \arw \PP^{n}$ be the morphism 
  defined by $[1: z^1: \cdots :z^n] \in \PP^{n}=\PP_{\! \! *} (Q^{\vee})$.
  We define $X' \subset \calu \subset \G(n,\PP^{n+1}) \times \PP^{n+1}$ to be the closure of the morphism   \[
  \nu : \Spec L \stackrel{(\nu_1,\nu_2)}{\longrightarrow} Y \times \PP^n \subset \G(n,\PP^{n+1}) \times \PP^{n} \stackrel{\bir}{\dashrightarrow} \calu,
  \]
  where the last birational map is obtained from $\Go \times \cP(Q\spcheck) = \sU|_{\Go}$.
  Then $\nu : \Spec L \arw X'$ is birational since $K(X')=K(Y)(z^1,\ldots,z^n ) =K(z^1,\ldots,z^n ) $ and 
  \[
  L \supset K(z^1,\ldots,z^n ) \supset K(z^{r+1} ,\ldots,z^n) =K(x_{r+1},\ldots,x_n) =L,
  \]
  where the last equality follows from the choice of $x_i$ (see the first paragraph of this proof).
  Since $z^1,\ldots,z^n$ form a separating transcendence basis of $L/\kk$ by \autoref{clm_choice_of_f},
  $\pr_2 : X' \arw \PP^{n+1}$ is separable and generically finite by the local description \ref{eq:pr_2-U-to-PN}.
  By \ref{eq_linear_condition} and the local description \autoref{lem_loc_description_phi},
  this $X'$ satisfies condition (ii) in \autoref{thm:GaussImage}
  and we obtain this theorem.
\end{proof}

\begin{rem}\label{thm_construction_for_insep:rem}
  Let $X \subsetneq \PP^{N}$ be a non-degenerate projective variety in
  $p = \chara\kk > 0$.
  Let $N_1 > N$ be an integer.
  Then we have non-degenerate $X_1 \subset \PP^{N_1}$ 
  such that
  $K(X_1)/K(\gamma_{X_1}(X_1))$ induced by $\gamma_{X_1}$
  coincides with $K(X)/K(\gamma_X(X))$,
  as follows.

  Let $n := \dim(X)$.
  Choosing suitable homogeneous coordinates on $\PN$, we may take a local parametrization
  $[1:z^1:\dots:z^n:f^{n+1}:\dots:f^{N}]$ of $X$ around $x = [1:0:\dots:0] \in X$,
  where $z^1,\dots,z^n$ form a system of parameters of
  the regular local ring $\sO_{X,x}$
  and $f^{n+1}, \dots, f^{N} \in \sO_{X,x}$.
  In this setting, we have 
  \begin{equation}\label{eq:KgammaX}
    K := K(\gamma_X(X)) = \kk \left(\Bigl\{f^j - \sum_{\Ri} z^i f^j_{z^i}\Bigr\}_{\Rj}, \set{f^j_{z^i}}_{\Rij} \right),
  \end{equation}
  where $f^j_{z^i} := \diff{f^j}{z^i}$
  (for example, see \cite[\textsection{}2]{fukaji2010}).
  Then we can define a rational map
  $\mu: X \dashrightarrow \PP^{N_1}$ parametrized by
  $[1:z^1:\dots:z^n:f^{n+1}:\dots:f^{N}:g^{N+1}:\dots:g^{N_1}]$
  with general $g^{N+1},\dots,g^{N_1} \in K^p$.
  By definition, $X$ is birational to $X_1:=\overline{\mu(X)} \subset \PP^{N_1}$.
  In addition, $K(\gamma_{X_1}(X_1))$ coincides with $K = K(\gamma_{X}(X))$ since
  $g^k_{z^i} = 0$.  
  Since $g^{N+1},\dots,g^{N_1} \in K^p$ are general,
  $X_1 \subset \PP^N$ is non-degenerate.
  Thus the statement follows.
\end{rem}

\section{Fibers of Gauss maps}
\label{sec_fibers}

Throughout this section, we assume $p = \chara \kk > 0$.
In \textsection{}3, we prove
\autoref{thm:anyY} by using \autoref{thm:GaussImage} and \autoref{lem_number_of_generators}.
\autoref{thm:family} is similarly proved by using \autoref{thm:GaussImage} and \autoref{lem_number_of_generators_geom} as follows.

\begin{proof}[Proof of \autoref{thm:family}]
  \begin{inparaenum}[\noindent\itshape Step 1.]
  \item 
    Let $\Frob : Y^{(1/p)} \arw Y$ be the Frobenius morphism
    and let $\calf' :=  Y^{(1/p)}  \times_{Y} \calf$ be the fiber product with the following diagram,
    \begin{equation*}
      \begin{split}
        \xymatrix{          \calf'  \ar[d]_(.45){f'} \ar[r]^(.5){\Frob'} & \calf  \ar[d]^{f}
          \\
          Y^{(1/p)} \ar[r]^{\Frob}   & Y,
        }      \end{split}
    \end{equation*}
    where $f'$ and $\Frob'$ are the induced morphisms.
    Note that  $\calf'$ is contained in $Y^{(1/p)} \times \PP^{N'}$
    since $\calf \subset Y \times \PP^{N'}$.
    For the first projection
    $\bar{f}': Y^{(1/p)} \times \PP^{N'} \rightarrow Y^{(1/p)}$,
    we have $f' = \bar{f}' |_{\calf'}$.

    Fix an embedding $Y^{(1/p)} \hookrightarrow \PP^{N''}$ and
    let
    \[
    \calf' \subset Y^{(1/p)}  \times \PP^{N'} \hookrightarrow \PP^{N''} \times \PP^{N'} \hookrightarrow \PP^{\overline{N}}
    \]
    be the Segre embedding.
    Let $\pi : \PP^{\overline{N}} \dashrightarrow \PP^n$ be a general linear projection.
    Then $\pi |_{\calf'}: \calf' \dashrightarrow \PP^n$ is separable and generically finite
    since $\dim \calf'=\dim \calf =n$.
    \\

  \item \label{item:n>Nd}
    Note that, if $n \geq N'$, then
    the restriction of $\pi$ on $\{y'\} \times \PP^{N'} \subset \PP^{\overline{N}}$
    is a linear embedding for general $y' \in Y^{(1/p)} $.
    \\

  \item 
    Let $\sU \subset \GN \times \PN$ be the universal family of $\GN$.
    By identifying $\PP^n$ with $\cP (Q\spcheck)$,
    we have a birational map from
    $\GN \times \PP^n$ to
    $\Go \times \cP(Q\spcheck) = \sU|_{\Go} \subset \sU$
    (see \autoref{sec:notation}).
    Let $\phi : Y^{(1/p)}  \times \PP^{N'} \dashrightarrow \calu \subset \G(n,\PP^N) \times \PP^N$ be the rational map defined as
    \[
    \phi : Y^{(1/p)}  \times \PP^{N'} \stackrel{(\Frob \circ \bar{f}' , \pi)}{\dashrightarrow} Y \times \PP^n  \subset \GN \times \PP^n 
    \stackrel{\bir}{\dashrightarrow} \calu.
    \]
    Set $X'\subset \sU$ to be the closure of $\phi(\calf')$.
    Since $\dim Y \geq 1$,
    we can apply \autoref{lem_number_of_generators_geom} to $\Frob \circ f' : \calf' \arw Y $. 
    Hence $\phi |_{\calf'} :  \calf' \dashrightarrow X'$ is birational.

    Since $ \pi |_{\calf'}: \calf' \dashrightarrow \PP^n$ is separable and generically finite,
    so is the second projection
    $\pr_2 : X' \arw \PN$
    (note that $\GN \times \PP^n \dashrightarrow \sU \rightarrow \PN$ is described
    on an open subset $\Go \times \PP^n \subset \GN \times \PP^n$
    by $((a_i^j), [z^0:z^1:\dots:z^n]) \mapsto [z^0:z^1:\dots:z^n:\cdots]$
    as in \ref{eq:pr_2-U-to-PN}, where the homogeneous coordinates $z^0,z^1,\dots,z^n$ on $\Pn$ appears in those on $\PN$).

    Since $\Frob \circ f'$ is of rank $0$,
    so is the first projection $\pr_1 :X' \arw Y$.
    Hence $X'$ satisfies condition \ref{thm:GaussImage:shr} in \autoref{thm:GaussImage}
    by \autoref{rem_Phi=0}.
    \\

  \item 
    Set $X := \pr_2(X') \subset \PP^N$.
    By \autoref{thm:GaussImage},
    we have the following commutative diagram.
    \[
    \xymatrix{      \calf' \ar@{-->}@/^1pc/@<1ex>[rr]^{\bir} \ar@{-->}[r]^(.6){\varphi} \ar@{->>}[rd]_{\Frob \circ f'  } & X' \ar[d]^(.4){\pr_1} \ar[r]^(.4){\pr_2} & X    \ar@{-->}[ld]^{\gamma_X} \ar@{}[r]|{\kern-1ex\mbox{$\subset$}}
      & \PN
      \\
      & Y . & 
    }    \]
    Since $X'$ is the graph of $\gamma_X$, $\pr_2$ is birational.
    Hence $\pr_2 \circ \phi : \calf' \dashrightarrow X$ is birational as well.

    Set $h := \Frob' \circ ( \pr_2 \circ \phi )^{-1} : X \dashrightarrow \calf$.
    Since $  ( \pr_2 \circ \phi )^{-1}$  is birational and $\Frob' $ is purely inseparable and finite,
    $h$ is purely inseparable and generically finite, i.e., $h$ is generically bijective.
    By the above commutative diagram, it holds that
    \[
    \gamma_X = (\Frob \circ f'  )\circ ( \pr_2 \circ \phi )^{-1} =( f \circ \Frob') \circ ( \pr_2 \circ \phi )^{-1} = f  \circ h.
    \]
    Hence the former statement of the theorem is proved.
    \\

  \item 
    Now assume $n \geq N'$.
    For general $y \in  Y$,
    let $y' \in Y^{(1/p)}$ be the unique point over $y$.
    Since the restriction
    $\pi  |_{\{y'\} \times \PP^{N'} }:  \{y'\} \times \PP^{N'}  \arw \PP^n$ appeared in Step~2 and
    \[
    \{y\} \times \PP^n \subset \GN \times \PP^n  \stackrel{\bir}{\dashrightarrow} \calu \stackrel{\pr_2}{\longrightarrow} \PP^N
    \]
    are linear embeddings,
    $F'_{y'}:=(f')^{-1}(y')_{red} \subset \{y'\} \times \PP^{N'}$ is projective equivalent to 
    $\overline{\pr_2 \circ \phi(F'_{y'})} =\overline{ \gamma_X^{-1}(y)}_{red} \subset \PP^N$.
    Since $F'_{y'} \subset \{y'\} \times \PP^{N'}$ is projectively equivalent to $F_y \subset \{y\} \times \PP^{N'}$,
    we have the latter statement of the theorem.
  \end{inparaenum}
\end{proof}

\begin{ex}\label{thm:family:rem}
  Assume $p >0$. We give some applications of \autoref{thm:family}
  by taking special $\calf \subset Y \times \PP^{N'}$
  with the surjective first projection $f: \calf \rightarrow Y$.
  We denote by $F_y$ the fiber of $f$ at $y \in Y$ with the reduced structure.

  \begin{inparaenum}
  \item 
    Let $Y \subset \G(n, \PP^N)$ and $F \subset \PP^{N'} $ be projective varieties with $\dim Y \geq 1$ and $\dim Y + \dim F =n$.
    Assume $n \geq N'$. Then, by taking $\calf = Y \times F$, we have $X \subset \PN$ such that the image of $\gamma$ equal to $Y$, and
    a general fiber of the Gauss map $\gamma$ of $X$ with the reduced structure is projectively equivalent to $F$.

    As mentioned in \autoref{sec:introduction},
    by (1), we can recover Fukasawa's construction of general fibers of $\gamma$ \cite[Theorem 1]{Fukasawa2006}, and also control the image of $\gamma$ at the same time.

  \item \autoref{thm:family} gives examples of $X$'s
    such that general fibers of $\gamma_X$ with the reduced structures are not isomorphic each other.
    For example, we may take $\calf\rightarrow Y$ to be an elliptic fibration with $F_{y} \not\simeq F_{y'}$ at general $y$ and $y'$.

    We note that Fukasawa \cite[Example~2.3]{Fukasawa2006-2} gave an explicit example of a Gauss map such that all isomorphism classes of elliptic curves appears as general fibers.
    \\

    Next we consider irreducible components of a general fiber of $\gamma$.
    Kaji \cite[Example 4.1]{Kaji1986} \cite{Kaji1989} and Rathmann \cite[Example~2.13]{Rathmann} gave examples of Gauss maps whose general fiber is a union of two or more points.
    In \cite[Remark~3.2]{Fukasawa-rims}, Fukasawa mentioned without proof
    that his construction can give $X$ such that a general fiber of $\gamma_X$ is the union of two or more $F$'s.
    We give generalizations of their results as follows.

    Let $s: Z \dashrightarrow Y \subset \GN$ be a separable, generically finite, dominant rational map with $\dim Y \geq 1$,
    and let $\calg \subset Z \times \PP^{N'}$ with $n \geq N'$ be an $n$-dimensional subvariety such that the first projection
    $g: \calg \rightarrow Z$ is surjective.
    We denote by $G_z$ the fiber of $g$ at $z \in Z$ with the reduced structure.

  \item
    We set $\calf \subset Y \times \PP^{N'}$
    to be the closure of the image of $\calg$
    under $s \times \id_{\PP^{N'}}$.
    Then $f: \calf \rightarrow Y$ satisfies
    the assumption of \autoref{thm:family}.
    Thus we can construct $X$ for this $\calf$.
    For general $y \in Y$, $F_y$
    is the union of fibers $G_z$ with $z \in s^{-1}(y)$.

    For example, by taking $\calg \rightarrow Z$ as an elliptic fibration as in (2),
    we can construct $X$ such that a general fiber of $\gamma$ of $X$ with reduced structure
    is the union of irreducible components which are not isomorphic each other.

  \item It may happen that $G_z = G_{z'} \subset \PP^{N'}$ with $z \neq z' \in s^{-1}(y)$ in (3). To deal with such cases, we modify the construction of (3) as follows.

    We take an embedding $Z \subset \PP^{N''}$ and the Segre embedding
    $\PP^{N''} \times \PP^{N'} \subset \PP^{\bar N}$, and consider
    a general linear projection $\pi:\PP^{\bar N} \dashrightarrow \Pn$.
    Then we set $\calf \subset Y \times \Pn$
    to be the closure of the image of $\calg$ under the induced map
    $Z \times \PP^{N'} \dashrightarrow Y \times \Pn$ sending
    $(z,a)$ to $(s(z),\pi(z,a))$.
    From \autoref{lem_number_of_generators_geom},
    this $\calf \subset Y \times \Pn$ is birational to $\calg$.
    Hence we can construct $X$ such that $F_y$ is the union of
    $\pi(G_z)$ with $z \in s^{-1}(y)$,
    where $\pi(G_z) \neq \pi(G_{z'}) \subset \Pn$ for $z \neq z'$.
    We note that $\pi(G_z)$ is projectively equivalent to $G_z$
    since $n \geq N'$.

  \item
    Let $\calf \subset Y \times \PP^{N'}$ 
    satisfy the assumption of \autoref{thm:family}.
    Applying (4) to $\calg = \calf \times_Y Z$,
    we have a new $\calf^s \subset Y \times \PP^{n}$
    such that $F^{s}_y$ at general $y \in Y$
    is the union of $c = \deg(s)$ copies of $F_{y}$.

    For example, setting $\calf = Y \times F$ in (1),
    we can construct $X$ for $\calf^s$
    such that the number of irreducible components of a general fiber of
    $\gamma$ of $X$ is equal to $c$;
    more precisely, the fiber is a union of $c$ copies of $F$.
    The fiber can be a \emph{disjoint} union
    of $c$ copies of $F$
    if $N' > 2\cdot \dim(F)$.
    
  \end{inparaenum}
\end{ex}

\begin{rem}\label{rem_non-degen}
  In \autoref{thm:family}, the projective variety
  $X \subset \PP^N$ is non-degenerate if so is
  the union $\bigcup_{[L] \in Y} L  \subset \PP^N$ for the subvariety $Y \subset \GN$.
  The reason is as follows:
  If $X \subset \PP^N$ is degenerate,
  i.e., $X$ is contained in some hyperplane $H \subset \PP^N$,
  then every embedded tangent space $\TT_xX$ is also contained in $H$.
  Since $Y = \overline{\gamma(X)}$, it follows that
  $\bigcup_{[L] \in Y} L  \subset H$.
\end{rem}

The rank of the Gauss map of $X$ which we construct in the proof of \autoref{thm:family} is zero.
Roughly,
the following proposition states that we can increase the ranks of Gauss maps
without changing general fibers.

\begin{prop}\label{prop_rank}\label{thm_fib_r_c}
  Let $X_1 \subset \PP^{N_1}$ be a non-degenerate projective variety
  and let $N, r \geq 0$ be integers such that $(p,r) \neq (2,1)$ and
  \begin{align*}    N - N_1 \geq  \left\{ 
      \begin{array}{cl} 
        r+1 & \text{if } p \geq 3, \text{ or } \ p=2, r : \text{even} ,\\ 
        r+2 & \text{if } p=2, r : \text{odd}.\\ 
      \end{array} \right.
  \end{align*}
  Then there exists a non-degenerate projective variety $X \subset \PP^N$ and a dominant rational map
  $\mu : \gamma_{X}(X) \dashrightarrow \gamma_{X_1} (X_1) $ 
  such that $\rank \gamma_X = \rank \gamma_{X_1} +r$ and
  the fiber ${\overline{\gamma_{X}^{-1}(y)}} \subset \PP^{N}$ of $\gamma_{X}$ over general $y \in \gamma_{X}(X) $
  is projectively equivalent to ${\overline{\gamma_{X_1}^{-1}(\mu(y))}} \subset \PP^{N_1}$.
\end{prop}

\begin{rem}
  In characteristic $2$,
  it is known that the rank of the Gauss map of any projective variety \emph{cannot} be equal to 1 (see \cite[Remark 5.3]{gmaptoric}, for example).
  This is the reason why we assume $(p,r) \neq(2,1)$ in the above proposition.
\end{rem}

\begin{rem}\label{prop_rank:rem}
  By combining \autoref{thm:family} and \autoref{prop_rank},
  we see that a general fiber of a Gauss map can be non-linear even if the rank of the Gauss map is non-zero.
  For example, 
  taking $X_1 \subset \PN$ as in \autoref{thm:family:rem} and
  applying \autoref{prop_rank},
  we can 
  increase the ranks of Gauss maps
  without changing general fibers.
\end{rem}

To show \autoref{thm_fib_r_c}, we prepare \autoref{rem:thm_fib_r_c:projs} below.
Let $N_1 < N$ be positive integers.
We consider a birational map
\begin{align*}
  \phi : \PP^{N_1} \times \PP^{N-N_1} &\dashrightarrow \PP^N
  \\
  \left( [1: u^1: \cdots : u^{N_1}] , [1: v^1: \cdots: v^{N-N_1}] \right) &\mapsto [1: u^1: \cdots : u^{N_1} : v^1: \cdots: v^{N-N_1}],
\end{align*}
where $[1: u^1: \cdots : u^{N_1}] \in \PP^{N_1}$
(resp. $[1: v^1: \cdots: v^{N-N_1}] \in \PP^{N-N_1}$)
is the affine coordinates on the standard affine subspace
$\A^{N_1} \subset \PP^{N_1}$
(resp. $\A^{N-N_1} \subset \PP^{N-N_1}$).  
We set  $\pi_i = p_i \circ \phi^{-1}  $ for $i=1,2$,
where $p_i$ is the projection from
$\PP^{N_1} \times \PP^{N-N_1}$ to the $i$-th factor.   
More explicitly,
\begin{align*}
  \pi_1 ([Z^0 : \cdots : Z^N]) & := [Z^0 : Z^1 : \cdots :Z^{N_1} ] \in \PP^{N_1} ,\\
  \ \pi_2 ([Z^0 : \cdots : Z^N]) &:= [Z^0 : Z^{N_1+1} \cdots :Z^{N}] \in \PP^{N-N_1} 
\end{align*}
for the homogeneous coordinate $[Z^0 : \cdots : Z^N] \in \PP^N$.

\begin{lem}\label{rem:thm_fib_r_c:projs}
  These linear projections $\pi_1$ and $\pi_2$ induce a natural birational map
  \[
  \psi : {\G(n-r, \PP^{N_1}) \times \G(r, \PP^{N-N_1})} \dashrightarrow \G(n,\PP^N)
  \]
  onto the image.
\end{lem}

\begin{proof}
  For a general $(n-r)$-plane $L_1 \subset \PP^{N_1}$
  and a general $r$-plane $L_2 \subset \PP^{N-N_1}$,
  the intersection 
  $\overline{\pi_1^{-1}(L_1)} \cap \overline{\pi_2^{-1}(L_2)} \subset \PN$ is a linear subvariety of codimension
  \[
  (N_1-n+r) + (N-N_1-r)=N-n,
  \]
  i.e., it is of dimension $n$.
  Hence we can define $\psi$ as a rational map by
  \[
  \psi ([L_1], [L_2] ) =\left[ \overline{\pi_1^{-1}(L_1)} \cap \overline{\pi_2^{-1}(L_2)} \right] \in \G(n,\PP^N).
  \]

  Next, in order to show $\psi$ is birational,
  we describe $\psi$ locally as follows. As in \autoref{sec:notation}, we consider standard affine open subsets
  $\Go_1 \subset \G(n-r, \PP^{N_1})$ and $\Go_2 \subset \G(r, \PP^{N-N_1})$, that is,
  $\Go_1$ (resp.\ $\Go_2$) is the set of $(n-r)$-planes $L_1 \subset \PP^{N_1}$ (resp.\ $r$-planes $L_2 \subset \PP^{N-N_1}$) spanned by the points corresponding to the rows of
  \[
  [E_{n-r+1} \  A_1] \quad (\text{resp. } [E_{r+1}  \  A_2])
  \]
  for an $(n-r+1) \times (N_1 -(n-r))$ matrix $A_1$ (resp.\ an $(r+1) \times (N-N_1 -r)$ matrix $A_2$),
  where $E_t$ is the $t \times t$ unit matrix. Recall that
  $\Go_i$ is identified with an affine space by sending $L_i$ to $A_i$ for $i=1,2$.
  Then $\overline{\pi_1^{-1}(L_1)} \cap \overline{\pi_2^{-1}(L_2)}$ is spanned by the $n+1$ points corresponding to the rows of
  \[
  \begin{bmatrix}
    1&  \bm 0 & \bm a_1 &   \bm 0  & \bm a_2
    \\
    ^{t} \bm 0 &  E_{n-r}  &   A'_1 & O  & O
    \\ 
    ^{t} \bm 0 &   O    & O &  E_{r}   & A'_2
  \end{bmatrix},
  \] 
  where $\bm 0$ and $O$ are the raw zero vector and zero matrix of suitable sizes,
  and $\bm a_i, A'_i$ are given by
  \[
  A_i=
  \begin{bmatrix}
    \bm a_i
    \\
    A'_i
  \end{bmatrix}.
  \]
  Now we can regard $\psi$ locally as a morphism from
  $\Go_1 \times \Go_2$ to $\A^{(n+1)(N-n)}$
  (the latter affine space corresponds to a suitable open subset of $\GN$),
  which sends $(A_1, A_2)$ to the $(n+1) \times (N-n)$ matrix
  \[
  \begin{bmatrix}
    \bm a_1 & \bm a_2
    \\
    A'_1 & O
    \\ 
    O  & A'_2
  \end{bmatrix}.
  \]
  Hence $\psi$ is birational onto the image. 
\end{proof}

\begin{proof}[Proof of \autoref{thm_fib_r_c}]
  By assumption,
  we can take a non-degenerate projective variety ${X_2} \subset \PP^{N - N_1}$ with $\dim {X_2} =r$
  such that the Gauss map $\gamma_{X_2}$ of ${X_2}$ is birational
  (for example, see \cite[Examples 3.9 and 3.10]{gmaptoric} and see also \autoref{thm_fib_r_c:rem} below).
  
  Now we use the notation of \autoref{rem:thm_fib_r_c:projs}.
  We may assume that $\psi$ is defined at general points of $\gamma_{X_1}(X_1) \times \gamma_{X_2}(X_2)$ and $\psi|_{\gamma_{X_1}(X_1) \times \gamma_{X_2}(X_2)}$ is birational onto its image,
  after suitable linear transformations on projective spaces $\PP^{N_1}, \PP^{N-N_1}$.
  We define $X \subset \PP^N$ to be the closure $\overline{\phi(X_1 \times X_2)}$.
  Then $\phi|_{X_1\times X_2} : X_1 \times X_2 \dashrightarrow X$ is birational and its inverse is
  $(\pi_1|_{X}, \pi_2|_{X})$.
  We note that the non-degeneracy of $X \subset \PP^N$ follows from that of $X_1$ and $X_2$ since
  $X$ and $X_1 \times X_2$ can be identified on the open subsets $\A^N \simeq \A^{N_1} \times \A^{N-N_1}$ by the definition of $\phi$.

  Let $x \in X$ be a general point and set $x_i := \pi_i(x) \in X_i$ ($i=1,2$).
  Since $(\pi_i \circ \phi) |_{X_1 \times X_2}= p_i|_{X_1 \times X_2}$,
  $\pi_i|_{X}: X \dashrightarrow X_i$ is separable.
  Thus we have $\pi_i(\TT_xX) = \TT_{x_i}X_i$.
  Since $\psi$ is defined at general points of
  $\gamma_{X_1}(X_1) \times \gamma_{X_2}(X_2)$,
  $\bigcap_{i=1,2} \overline{\pi_i^{-1}(\TT_{x_i}X_i)}$ is of dimension $n$,
  which means that this intersection coincides with $\TT_xX$.
  In other words, $\gamma_X(x) = \psi(\gamma_{X_1}(x_1), \gamma_{X_2}(x_2))$ holds.
  Therefore, we have the following commutative diagram
  \begin{equation}\label{eq:diag-phi2}
    \begin{split}
      \xymatrix{        X  \ar@{-->}[r]^(.45){\gamma_X} & \G(n,\PP^N) 
        \\
        X_1 \times X_2 \ar@{-->}[u]^{\phi|_{X_1\times X_2}} \ar@{-->}[r]^{\kern-5em\gamma_{X_1} \times \gamma_{X_2}} & \G(n-r, \PP^{N_1}) \times \G(r, \PP^{N-N_1}) . \ar@{-->}[u]^{\psi}
      }    \end{split}
  \end{equation}
  Since $\phi|_{X_1\times X_2}$ is birational,
  $\overline{\gamma_X(X)} = \overline{\psi (\gamma_{X_1}(X_1) \times \gamma_{X_2}(X_2)) }$ holds.
  Since $\psi|_{\gamma_{X_1}(X_1) \times \gamma_{X_2}(X_2)}$ is birational onto the image,
  we have the inverse birational map
  \[
  \psi|_{\gamma_{X_1}(X_1) \times \gamma_{X_2}(X_2)}^{-1} :
  \gamma_X(X) \dashrightarrow \gamma_{X_1}(X_1) \times \gamma_{X_2}(X_2).
  \]
  Let $\mu : \gamma_X(X) \dashrightarrow \gamma_{X_1}(X_1)$ be the composite map of 
  $ \psi|_{\gamma_{X_1}(X_1) \times \gamma_{X_2}(X_2)}^{-1} $ and the first projection
  $\gamma_{X_1}(X_1) \times \gamma_{X_2}(X_2) \arw \gamma_{X_1}(X_1)$.

  Since $\gamma_{X_2} ,\phi |_{X_1 \times X_2},\psi|_{\gamma_{X_1}(X_1) \times \gamma_{X_2}(X_2)}$ are birational,
  we have $\rank \gamma_X =\rank( \gamma_{X_1} \times \gamma_{X_2}) = \rank \gamma_{X_1} +r$.
  
  Let $(y_1,y_2) \in \gamma_{X_1}(X_1) \times \gamma_{X_2}(X_2)$ be a general point.
  Since $\gamma_{X_2}$ is birational,
  $\gamma_{X_2}^{-1}(y_2)$ consists of one point $x_2 \in X_2$.
  Hence the fiber of $\gamma_{X_1} \times \gamma_{X_2}$ over $(y_1,y_2)$ is 
  $\gamma_{X_1}^{-1}(y_1) \times \{x_2\} \subset X_1 \times X_2 \subset \PP^{N_1} \times \PP^{N_2}$.
  By the diagram \ref{eq:diag-phi2},
  $\phi (\gamma_{X_1}^{-1}(y_1) \times \{x_2\}) =\gamma_X^{-1}(y)$ for $y:=\psi(y_1,y_2)$.
  Since $\phi |_{\PP^{N_1} \times \{x_2\}} : \PP^{N_1} \times \{x_2\} \arw \PP^N$ is a linear embedding by the definition of $\phi$,
  $\gamma_{X_1}^{-1}(y_1) \times \{ x_2\} \subset \PP^{N_1} \times \{x_2\}$ is projectively equivalent to $\gamma_X^{-1}(y) \subset \PP^N$.
  Since $y_1=\mu(y)$, this proposition is proved.
\end{proof}

\begin{rem}\label{thm_fib_r_c:rem}
  In the proof of \autoref{thm_fib_r_c},
  we can indeed take $X_2 \subset \PP^{N-N_1}$ whose function field is any given field $K_2$ of transcendental degree $r$ over $\kk$
  such that the Gauss map $\gamma_{X_2}$ is birational.
  Hence $X$ in the statement of \autoref{thm_fib_r_c} can be birational to $X_1 \times \Spec K_2$.

  For example, in the case when $p=2$, $r$ is odd, and $N-N_1 = r+2$,
  we take a separating transcendence basis $\set{z^1,\dots,z^r}$ of $K_2/\kk$ and
  take $f \in K_2$ such that $K_2=\kk(z^1,\dots,z^r)(f)$.
  Then we set $X_2$ to be the closure of the image of
  $\Spec(K_2) \rightarrow \PP^{r+2}$ given by
  \[
  [1:z^1:z^2:\dots:z^r:t\cdot f+ z^1z^2+z^3z^4+\dots+z^{r-2}z^{r-1}: z^{r-1}z^{r}]
  \]
  with general $t \in \kk$. From \ref{eq:KgammaX}, we find that $\gamma_{X_2}$ is separable and generically finite, and hence birational by \cite[Theorem~1]{expshr}.
\end{rem}

\section{Degeneracy maps and the second fundamental form}
\label{sec:degen-maps-second}

In this section,
we give an expression of degeneracy maps in the context of shrinking maps
by using the {second fundamental form}.
We work over an algebraically closed field of arbitrary characteristic.

Let $\gamma = \gamma_X: X \dashrightarrow \GN$ be the Gauss map of $X \subset \PN$.
We recall that $d_x\gamma: T_xX \rightarrow T_{\gamma(x)}\GN$
is the differential of $\gamma$.
Note that there is a one-to-one correspondence between $l$-planes in $\TT_xX$
containing $x$
and $l$-dimensional vector subspaces in $T_xX$.

\begin{defn}\label{thm:def-kappa}
  A \emph{degeneracy map} $\kappa$ of $X$ is defined as
  a rational map from $X$ to a Grassmann variety which sends a general point $x \in X$ to
  the linear subvariety in $\TT_xX$ corresponding to $\ker(d_x\gamma) \subset T_xX$.
  This $\kappa$ is also described as follows.

  For the differential $d \gamma : T_{X^{sm}} \rightarrow \gamma^*T_{\GN}$
  and the exact sequence
  \begin{equation}\label{eq:O-Q1-TX}
    0 \rightarrow \sO_{X^{sm}} \rightarrow \gamma^*\sQ\spcheck(1)  \xrightarrow{\theta(1)} T_{X^{sm}} \rightarrow 0
  \end{equation}
  (see the diagram~\ref{eq:QLS-TTN} in the proof of \autoref{thm:GaussImage}),   we can take an open subset $\Xo \subset X$ such that
  $\ker(d\gamma \circ \theta(1))|_{\Xo}$
  is a subbundle of   $\gamma^*\sQ\spcheck(1)|_{\Xo} \subset H^0(\PN, \sO(1)) \otimes \sO_{\Xo}(1)$,
  which induces a morphism from $\Xo$ to a Grassmann variety due to universality.
  This morphism is nothing but $\kappa$.

  See \cite{FP}   for properties of $\kappa$ in the characteristic zero case.
\end{defn}

\begin{prop}\label{thm:kap-defby-shr}
  Let $X \subset \PN$ be a projective variety and
  let $X' \subset \GN \times \PN$ be the graph of the Gauss map $\gamma$ of $X$.
  Then the following three maps coincide:

  \begin{enumerate}
  \item the degeneracy map $\kappa$ of $X$,

  \item the shrinking map $\sigma_{X, \gamma}: X \dashrightarrow \Gr(n^{-}, N)$
    of $X$ with respect to the Gauss map $\gamma$ of $X$,

  \item the composite map $\sigma_{X', \pr_1} \circ (\gamma, \id_X)$,
    where $\sigma_{X', \pr_1}$ be the shrinking map of $X'$ with respect to $\pr_1: X' \rightarrow \GN$.
  \end{enumerate}
\end{prop}

In order to prove \autoref{thm:kap-defby-shr},
we discuss the second fundamental form in terms of sheaf homomorphisms.
Let $\Xo \subset X$ be an open subset consisting of smooth points of $X$
and regard $\gamma$ as $\gamma|_{\Xo}$.
As in \autoref{sec:subv-univ-family}, we denote by $\sQ$ and $\sS$
the universal quotient bundle and subbundle on $\GN$ with the exact sequence \ref{eq_univ_seq_U}.
The differential
$d\gamma: T_{\Xo} \rightarrow \gamma^*T_{\GN} = \gamma^*\sHom(\sQ\spcheck, \sS\spcheck)$ corresponds to a homomorphism
\[
\widetilde{d\gamma}: T_{\Xo} \otimes \gamma^*\sQ\spcheck \rightarrow \gamma^*\sS\spcheck
\]
under the identification $\Hom(T_{\Xo}, \gamma^*\sHom(\sQ\spcheck, \sS\spcheck))
\simeq \Hom(T_{\Xo} \otimes \gamma^*\sQ\spcheck, \gamma^*\sS\spcheck)$.
Then we can check that
the composition of
$T_{\Xo} \otimes \sO_{\Xo} \hookrightarrow T_{\Xo} \otimes \gamma^*\sQ\spcheck(1)$ induced by  \ref{eq:O-Q1-TX}
and
$\widetilde{d\gamma}(1): T_{\Xo} \otimes \gamma^*\sQ\spcheck(1) \rightarrow \gamma^*\sS\spcheck(1)$
is the zero map. 
Hence a homomorphism, called \emph{the second fundamental form},
\[
\tau: T_{\Xo} \otimes T_{\Xo} \rightarrow \gamma^*\sS\spcheck(1)
\]
is induced.
By definition, $\widetilde{d\gamma}(1)$ factors through $\tau$.

\begin{rem}
  It is known that $\tau$
  is symmetric.
  For instance, this can be shown by taking local parametrization of $X$ as in
  \autoref{sec:notation} and \autoref{sec:equality-between-two}.
\end{rem}

We define $\tau_i: T_{\Xo} \rightarrow \sHom(T_{\Xo}, \gamma^*S\spcheck)(1)$ for $i=1,2$ by
\[
\tau_1(x) = [T_{\Xo} \rightarrow \gamma^*S\spcheck(1): y \mapsto \tau(x, y)],\
\tau_2(x) = [T_{\Xo} \rightarrow \gamma^*S\spcheck(1): y \mapsto \tau(y, x)].
\]
In fact, $\tau_1 = \tau_2$ because of the symmetry of $\tau$.

\begin{proof}[Proof of \autoref{thm:kap-defby-shr}]
  Since $(\gamma, \id_X): X \dashrightarrow X'$ is birational,
  $\sigma_{X, \gamma} = \sigma_{X', \pr_1} \circ (\gamma, \id_X)$
  follows from \autoref{thm:shr-map-sep-map}.
  Now we show $\kappa = \sigma_{X, \gamma}$.

  The homomorphism $d\gamma$ coincides with the composition of
  $\tau_1$ and the injection
  \[
  \sHom(T_{\Xo}, \gamma^*\sS\spcheck)(1) = \sHom(T_{\Xo}(-1), \gamma^*\sS\spcheck)
  \hookrightarrow \sHom(\gamma^*\sQ\spcheck, \gamma^*\sS\spcheck),
  \]
  which is induced from the surjection $\theta(1): \gamma^*\sQ\spcheck(1) \twoheadrightarrow T_{\Xo}$ in \ref{eq:O-Q1-TX}.
  In particular,
  we have $\ker (d \gamma)=\ker (\tau_1) \subset T_{X^{\circ}}$.
  Thus the degeneracy map $\kappa$
  is induced by the subsheaf 
  \[
  \theta^{-1} (\ker (d \gamma) (-1)) = \theta^{-1} (\ker (\tau_1)(-1))
  \]
  of $\gamma^*\sQ\spcheck \subset H^0(\PP^N,\sO(1))^{\vee} \otimes \sO_{X^{\circ}}$.
  On the other hand,
  the homomorphism $\Phi: \gamma^*\sQ\spcheck \rightarrow \sHom(T_{\Xo}, \gamma^*\sS\spcheck)$
  given in \autoref{thm:def-shr-Zf}
  corresponds to $\widetilde{d\gamma}$
  under the identification
  $\Hom(\gamma^*\sQ\spcheck, \sHom(T_{\Xo}, \gamma^*\sS\spcheck)) \simeq
  \Hom(T_{\Xo} \otimes \gamma^*\sQ\spcheck, \gamma^*\sS\spcheck)$.
  This implies that $\Phi(1)$ coincides with the composition of
  $\gamma^*\sQ\spcheck(1) \twoheadrightarrow T_{\Xo}$ and $\tau_2$.
  Hence $\ker \Phi = \theta^{-1} (\ker (\tau_2 )(-1))$ holds.
  Since the shrinking map $\sigma_{X,\gamma}$ is induced by $\ker \Phi  \subset H^0(\PP^N,\sO(1))^{\vee} \otimes \sO_{X^{\circ}}$,
  and since $\tau_1=\tau_2$ holds,
  we have $\kappa = \sigma_{X,\gamma}$.
\end{proof}

Combining \autoref{thm:GaussImage} and \autoref{thm:kap-defby-shr},
we recover the main theorem of the first author's paper \cite{expshr}.

\begin{cor}[{\cite[Theorem~3.1]{expshr}}]\label{thm:char-sepa}
  Let $X \subset \PN$ be an $n$-dimension projective variety,
  and let $Y \subset \GN$ be a projective variety.
  We set $\Gamma(X) \subset \calu_{\G(n,\P^N)} \subset \GN \times \PN$ to be the graph of the Gauss map $\gamma = \gamma_X$ of $X$ (see \autoref{sec:subv-univ-family} for definition),
  and set $\sigma_Y: Y \dashrightarrow {\Gr(n^{-},\PN)}$ to be the shrinking map of $Y$ with respect to $Y \hookrightarrow \GN$, where $n^{-} = n^{-}_{\sigma_Y} \leq n$.
  Then the following are equivalent:

  \begin{enumerate}
  \item
    $\gamma: X \dashrightarrow \GN$ is separable,
    and $Y = \overline{\gamma(X)}$.

  \item $\Gamma(X) = \sigma_Y^*{\calu}_{\Gr(n^{-},\PN)}$ in $\GN \times \PN$.

  \item The second projection $\sigma_Y^*{\calu}_{\Gr(n^{-},\PN)} \rightarrow \PN$
    is separable and generically finite onto its image,
    and the image is equal to $X$
    (in particular, its dimension is $n$).
  \end{enumerate}
\end{cor}

Note that, we have
the linearity of a general fiber of a separable Gauss map 
by the implication (1) $\Rightarrow$ (2)(e.g., see \cite[Proof of Corollary 3.7]{expshr}),
and have a characterization of images of separable Gauss maps by the equivalence (1) $ \Leftrightarrow $ (3) (see \cite[Corollary 3.15]{expshr}).

\begin{proof}[Proof of \autoref{thm:char-sepa}]
  (1) $\Rightarrow$ (2):
  Let $\kappa: X \dashrightarrow \G(n^-_{\kappa}, \P^N)$
  be the degeneracy map,
  where $n^{-}_{\kappa} =\rank ( \ker d \gamma )  $ for $d \gamma : T_X \arw \gamma^* T_Y$.
  Since $d \gamma$ is surjective,
  we have $n^{-}_{\kappa} = n- \dim Y$.
  From \autoref{thm:kap-defby-shr}, we have
  $\kappa = \sigma_{X,\gamma}$.
  Since $\gamma$ is separable, \autoref{rem_shrinking_map} implies
  $\sigma_{X,\gamma} =\sigma_Y \circ  \gamma$; hence
  $\kappa = \sigma_Y \circ  \gamma$ and $n^{-}_{\kappa} = n^{-}_{\sigma_Y} = n^{-}$ hold.
  Thus we have $n^-=n- \dim Y$.

  Set $X' = \Gamma(X)$.
  Since $\pr_1 : X' \twoheadrightarrow Y$ is separable,
  $\sigma_{X',\pr_1} = \sigma_{Y} \circ \pr_1 $ holds.
  By \autoref{thm:GaussImage},
  it holds that $X' \subset \sigma_Y^{*} \calu_{\G(n^-,\P^N)}$.
  Since 
  \[
  \dim \sigma_Y^{*} \calu_{\G(n^-,\P^N)} = \dim Y + n^- = n= \dim X',
  \]
  $X'$ coincides with $\sigma_Y^{*} \calu_{\G(n^-,\P^N)}$.

  (3) $ \Rightarrow$ (2):
  Set $X' = \sigma_Y^{*} \calu_{\G(n^-,\P^N)}$.
  By assumption, $X' \rightarrow X$ is separable and generically finite.
  In particular, we have $\dim X' = n$.
  Since $\pr_1 : X' \arw Y$ is a projective bundle,
  it is separable; hence $\sigma_{X',\pr_1} =\sigma_Y \circ \pr_1$ holds and then the condition (iii) of \autoref{thm:GaussImage} satisfied.
  Thus \autoref{thm:GaussImage} implies (2).

  The implications (2) $\Rightarrow$ (1) and (2) $\Rightarrow$ (3)
  follow immediately.
\end{proof}

\end{document}